\documentclass[a4paper]{article}

\usepackage{amsmath}
\usepackage{amssymb}
\usepackage{amsthm}
\usepackage{mathrsfs}
\usepackage{wasysym}
\usepackage{bbm}		%
\usepackage[inline]{enumitem}
\usepackage{nccmath}	%
\usepackage{mathtools}
\usepackage{braket}		%
\usepackage[normalem]{ulem}		%
\usepackage{nicefrac}
\usepackage{cancel}
\usepackage{xcolor}
\usepackage[a4paper, left=2.7cm, right=2.7cm, top=3.5cm, bottom=5cm]{geometry}
\usepackage[pdftex=true,hyperindex=true,pdfborder={0 0 0}]{hyperref}
\usepackage[nottoc]{tocbibind}	%
\usepackage{doi}
\usepackage[numbers]{natbib}
\usepackage{cleveref}
\usepackage{tikz}
\usetikzlibrary{shapes,arrows,automata,matrix,fit,patterns}
\usepackage{soul} %

\usepackage{autonum}		
\theoremstyle{plain}
\newtheorem{theorem}{Theorem}[section]
\newtheorem{lemma}[theorem]{Lemma}

\newtheorem{proposition}[theorem]{Proposition}

\newtheorem{question}[theorem]{Question}
\theoremstyle{definition}
\newtheorem{definition}[theorem]{Definition}
\newtheorem{remark}[theorem]{Remark}
\newtheorem{example}[theorem]{Example}

\newcommand{\ZZ}{\mathbb{Z}}			%
\newcommand{\NN}{\mathbb{N}}			%
\newcommand{\RR}{\mathbb{R}}			%
\newcommand{\symb}[1]{\mathtt{#1}}		%
\newcommand{\isdef}{\coloneqq}			%
\DeclarePairedDelimiter\norm{\lVert}{\rVert}	%
\makeatletter
\newcommand{\vast}{\bBigg@{4}}
\newcommand{\Vast}{\bBigg@{5}}

\makeatother
\newcommand{\Var}{%
	\operatorname{\mathrm{Var}}%
}	

\newcommand{\Prob}{%
	\operatorname{\mathrm{Prob}}%
}
\newcommand{\ag}{\mathcal{A}}
\newcommand{\cT}{\mathcal{T}}           %
\newcommand{\dd}{\mathrm{d}}			%

\newcommand{\Sym}{{\rm Sym}}
\newcommand{\Map}{{\rm Map}}

\newcommand{\define}[1]{\emph{#1}}

\newcommand{\xConfig}[1]{%
	\begin{tikzpicture}[
		baseline=-\the\dimexpr\fontdimen22\textfont2\relax,ampersand replacement=\&]
		\matrix[
		matrix of math nodes,
		nodes={
			minimum size=1.4ex,text width=1.4ex,
			text height=1.4ex,inner sep=3pt,draw={gray!20},anchor=center
		}, row sep=1pt,column sep=1pt
		] (config) {#1};
		\node[draw,rectangle,help lines,gray!50, dashed,fit=(config),inner sep=-1pt] {};
	\end{tikzpicture}
}

\newtheorem{maintheorem}{Theorem}

\title{%
	The Lanford--Ruelle theorem for actions of sofic groups
}

\author{Sebasti\'an Barbieri and Tom Meyerovitch}

\newcommand{\Addresses}{{
		\bigskip

		\hskip-\parindent   S.~Barbieri, \textsc{Departamento de Matem\'{a}tica y ciencia de la computaci\'{o}n, Universidad de Santiago de Chile, Santiago. Chile.}\par\nopagebreak
		\textit{E-mail address}: \texttt{sebastian.barbieri@usach.cl}
		
		\medskip
		
		\hskip-\parindent   T.~Meyerovitch, \textsc{Department of Mathematics, Ben-Gurion University of the Negev, Be'er Sheva, Israel}\par\nopagebreak
		\textit{E-mail address}: \texttt{mtom@bgu.ac.il}
}}

\date{}

\begin{document}
	
	\maketitle
	
	\begin{abstract}
		Let $\Gamma$ be a sofic group, $\Sigma$ be a sofic approximation sequence of $\Gamma$ and $X$ be a $\Gamma$-subshift with nonnegative sofic topological entropy with respect to $\Sigma$. Further assume that $X$ is a shift of finite type, or more generally, that $X$ satisfies the topological Markov property. We show that for any sufficiently regular potential $f \colon X \to \RR$, any translation-invariant Borel probability measure on $X$ which maximizes the measure-theoretic sofic pressure of $f$ with respect to $\Sigma$, is a Gibbs state with respect to $f$. This extends a classical theorem of Lanford and Ruelle, as well as previous generalizations of Moulin Ollagnier, Pinchon, Tempelman and others, to the case where the group is sofic. 
		
		As applications of our main result we present a criterion for uniqueness of an equilibrium measure, as well as sufficient conditions for having that the equilibrium states do not depend upon the chosen sofic approximation sequence. We also prove that for any group-shift over a sofic group, the Haar measure is the unique measure of maximal sofic entropy for every sofic approximation sequence, as long as the homoclinic group is dense.
		
		On the expository side, we present a short proof of Chung's variational principle for sofic topological pressure. 
		
		\medskip
		
		\noindent
		\emph{Keywords: sofic groups, equilibrium states, Gibbs measures, sofic entropy, sofic pressure.}
		
		\smallskip
		
		\noindent
		\emph{MSC2020:}
		37B10,	%
		37D35,	%
		82B20. 	%
	\end{abstract}

	\section{Introduction}
	The classical Lanford--Ruelle theorem~\cite{LanfordRuelle1969} asserts that given a $\ZZ^d$-shift of finite type $X$ and an absolutely-summable, translation-invariant interaction $\Phi$ on $X$, any translation-invariant Borel probability measure which maximizes the \define{measure-theoretic pressure} with respect to $\Phi$ (that is, any \define{equilibrium state} for the interaction $\Phi$) is a \define{Gibbs state} with respect to $\Phi$. 
	
	Various aspects of the Lanford--Ruelle theorem have been generalized by a number of authors. It has been proven that any measure of maximal entropy for a $\ZZ^d$-shift of finite type is invariant with respect to the asymptotic relation~\cite{BurtonStief1994}. Also, the Lanford--Ruelle theorem has been extended to the context where the acting group is any countable amenable group, not only $\ZZ^d$~\cite{OllagnierPinchon1981,tempelman1984specific}, and even more generally, to subshifts which satisfy a condition called the \define{topological Markov property} and with respect to any random \define{environment}~\cite{BGMT_2020}. For an overview of these generalizations, we refer the reader to the introduction of~\cite{BGMT_2020}.

	A fair amount of work related to Gibbs measures over finite graphs and their limits has been carried out by the communities of mathematical physics and probability. In particular, Gibbs measures on tree-like graphs such as $d$-regular expanders have been considered in~\cite{Sly2010,AndreaMosselSly2012}.
	The emergence of sofic entropy~\cite{Bowen2010_2} and subsequent introduction of sofic pressure~\cite{chung_2013} provides a natural framework to study equilibrium measures and their relation with Gibbs states in the context of actions of sofic groups. Let us mention that this class of groups includes all countable amenable and residually finite groups, in particular, it includes all free groups. 
	
	In this paper we fully generalize the Lanford--Ruelle theorem to the context where the lattice is an arbitrary countable sofic group $\Gamma$ and the space of configurations $X$ has hard constraints that generalize the condition of being of finite type. More precisely, our theorem works for any subshift $X \subset \ag^{\Gamma}$ with nonnegative sofic topological entropy which satisfies the topological Markov property.

	\begin{maintheorem}{\label{thm:LR_norm_summable_interaction}}
		Let $\Sigma$ be a sofic approximation sequence for $\Gamma$, $X$ be a subshift that satisfies the topological Markov property such that $h_{\Sigma}(\Gamma \curvearrowright X)\geq 0$, $\Phi$ an absolutely-summable interaction on $X$ and $\mu$ an equilibrium measure on $X$ for $\Phi$ with respect to $\Sigma$. Then $\mu$ is Gibbs with respect to $\Phi$.
	\end{maintheorem}
	
	Some authors have already explored results that are related to the Lanford--Ruelle theorem beyond amenable groups. For instance, Grigorchuk and Stepin~\cite{GrigorchukStepin1984} provide a sketch of proof for the existence of Gibbs measures on the full $\Gamma$-shift on residually finite groups for locally-constant potentials. Many years later Alpeev~\cite[Theorem $1$]{Alpeev2016}, showed that for any locally-constant potential on the full $\Gamma$-shift (that is, no ``hard constraints'' or ``forbidden patterns'') of a sofic group $\Gamma$, there exist Gibbs measures and that furthermore, under certain conditions ensuring uniqueness, the value of a variant of sofic topological pressure does not depend upon the sofic approximation sequence. Recent results of Shriver~\cite{shriver2020free} essentially show that the conclusion of the Lanford--Ruelle theorem for the full $\Gamma$-shift holds with respect to finite-range interactions. Strictly speaking, Shriver's paper considers only a certain class of finitely generated sofic groups, although this restriction is presumably just for simplicity of exposition, where the generalization beyond full-shifts and beyond bounded range interactions does require further arguments. We should mention that Shriver's paper obtains additional related results related to invariance with respect to the Glauber dynamics, which imply Gibbsianness and are outside the scope of this work. Related objects have also been considered in~\cite{AustinPodder2018}, although the relation to sofic entropy in this work is somewhat implicit.
	
	We deduce~\Cref{thm:LR_norm_summable_interaction} and a close variant of it (\Cref{thm:LR_SV}) as instances of the following slightly more general result:
	
	\begin{maintheorem}{\label{thm:LR_etale_good}}
		Let $\Sigma$ be a sofic approximation sequence for $\Gamma$, $X$ be a subshift such that $h_{\Sigma}(\Gamma \curvearrowright X)\geq 0$, and let $f \in C(X)$ be a  $\rho_{\mathcal{T}^0(X)}$-limit of locally constant functions. Then any equilibrium measure $\mu$ on $X$ for $f$ with respect to $\Sigma$ is \'etale Gibbs with respect to $f$. 
	\end{maintheorem}
	
	We remark that~\Cref{thm:LR_etale_good} does not require any structural assumption on the subshift but provides the conclusion of the measure being ``étale Gibbs'' instead of just Gibbs. Let us explain how this statement is more general than~\Cref{thm:LR_norm_summable_interaction} and \Cref{thm:LR_SV}:%
	
	\begin{enumerate}
		\item As in \cite{Meyerovitch_2013}, the statement of~\Cref{thm:LR_etale_good} applies meaningfully to shifts even without the topological Markov property. The condition of being a Gibbs measure on $X$ with respect to a suitable map $f\colon X \to \RR$ can be stated in terms of the Radon-Nikod\'ym cocycle of the measure with respect to the relation $\cT(X)$ induced by the asymptotic pairs on $X$. Following ideas from~\cite{Meyerovitch_2013}, we define a subrelation $\cT^{0}(X)$ which we call the \define{\'etale asymptotic relation} (called the topological Gibbs relation in~\cite{Meyerovitch_2013}) and define \define{\'etale Gibbs measures} as those which satisfy the Gibbs condition with respect to this subrelation. It turns out that both notions of ``Gibbs'' coincide precisely when the subshift satisfies the topological Markov property (\Cref{prop:GibbsTMP}). We show  that the conclusion of the Lanford--Ruelle theorem holds for any subshift as long as we replace the condition of being Gibbs by being \define{\'etale Gibbs}.
		\item The statement of~\Cref{thm:LR_etale_good} applies to a larger class of functions. More precisely, a sequence $(f_n)_{n \in \NN}$ of locally constant functions converges in the $\rho_{\cT^{0}(X)}$-metric to $f \in C(X)$ if it converges uniformly, and furthermore, the cocycles induced by the functions $f_n$ on the asymptotic Gibbs relation $\cT^{0}(X)$ converge to the cocycle induced by $f$ on $\cT^{0}(X)$. We show that functions generated by absolutely-summable interactions (\Cref{thm:LR_norm_summable_interaction}) and functions which have summable-variation with respect to a filtration of $\Gamma$ (\Cref{thm:LR_SV}) satisfy this property.
	\end{enumerate}

	An interesting observation related to our main result is that, while for non-amenable sofic groups the notion of equilibrium measure depends upon a sofic approximation sequence, Gibbs measures do not. Therefore, if by some means one can show that a subshift $X \subset \ag^{\Gamma}$ with the topological Markov property admits a unique Gibbs measure $\mu$, then necessarily for every sofic approximation sequence $\Sigma$ such that $h_{\Sigma}(\Gamma \curvearrowright X)\geq 0$, we have that $\mu$ is translation-invariant and is the unique such measure which maximizes the measure-theoretic sofic entropy $h_{\Sigma}(\Gamma \curvearrowright X) = h_{\Sigma}(\Gamma \curvearrowright X,\mu)$. We prove this in~\Cref{thm:uniquegibbsmeasure}.
	
	In~\cite[Question 5.4]{chung_2013} it was asked whether for every single-site potential there is a unique equilibrium state on the full-shift over a countable sofic group. This question was answered positively in~\cite[Example 7]{Bowen_2020} and in \cite[Corollary 3.6]{Seward2019b}. Our results also show that the answer is positive but through the fact that there exists a unique translation-invariant Gibbs measure in that case.

	Establishing a property of equilibrium measures that is independent of the sofic approximation sequence is particularly relevant given a recent result of Airey, Bowen and Lin~\cite{AirBowLin2019} which produces a strongly irreducible subshift of finite type (and thus with the topological Markov property) on a sofic group $\Gamma$ for which two different sofic approximation sequences $\Sigma_1$ and $\Sigma_2$ of $\Gamma$ yield two distinct positive values of sofic topological entropy. If one were able to adapt their construction and produce a system with a unique Gibbs measure, then one would actually be able to show that a single measure can produce two distinct positive values of measure-theoretic sofic entropy with respect to distinct sofic approximation sequences, a problem which is still open~\cite[Remark 2]{AirBowLin2019}.
	
	Another application of our result concerns the existence of Gibbs measures. For any sofic group, any subshift of finite type which has nonnegative sofic topological entropy with respect to some sofic approximation sequence admits at least one translation-invariant equilibrium state with respect to any translation-invariant, absolutely-summable interaction (\Cref{thm:existenceGibbsmeasure}), in particular, our result shows that these equilibrium states are Gibbsian. The previous statement generalizes the aforementioned results of Grigorchuk and Stepin~\cite{GrigorchukStepin1984} and Alpeev~\cite[Theorem $1$]{Alpeev2016}. In fact, in our result the existence of translation-invariant Gibbs states holds for any subshift with the topological Markov property, which is more general than the finite type condition. We remark that the assumption of having finite sofic topological entropy holds trivially in the case where the acting group is amenable, but in the non-amenable case,  the assumption cannot be completely removed. In particular, for any non-amenable group there exist subshifts of finite type that admit no invariant measures, in which case the sofic topological entropy is equal to $-\infty$ for any sofic approximation sequence.

	A further application of our theorem can be given whenever a subshift carries an additional algebraic structure. More precisely, a \define{group shift} is a subshift whose alphabet is a finite group and which carries a group structure induced by the alphabet. These algebraic subshifts always admit the Haar measure as a translation-invariant Borel probability measure. We are able to show that for any group shift defined on a sofic group for which its homoclinic group is dense (the group of elements in the subshift which are asymptotic to the identity) then the Haar measure is the unique measure of maximal sofic entropy (\Cref{thm:groupshifts}) with respect to every sofic approximation sequence for $\Gamma$.
	
	\textbf{Acknowledgements}: The authors wish to thank Raimundo Brice\~no for  drawing our attention to Shriver's work, Hanfeng Li for very helpful comments on a first version of this article and Andrei Alpeev for suggesting reference~\cite{GrigorchukStepin1984}. We are also grateful to an anonymous referee for several helpful comments. S. Barbieri was supported by the FONDECYT grant 11200037 and T. Meyerovitch was supported by ISF grant 1058/18.

	\subsection{Organization of the paper}
	In~\Cref{sec:preliminaries} we provide all the necessary definitions and prove a few elementary results which relate the asymptotic relation with the topological Markov property.
	
	In~\Cref{sec:soficgroupspressure} we recall the definition of sofic groups and sofic pressure, and provide proofs of a few elementary structural results. We also provide a short self-contained proof of Chung's variational principle (\Cref{thm:variational}). 
	
	The proof of our main theorem in the case where the function is locally constant is given in~\Cref{sec:prooflocal}.
	The heart of the argument is an estimate for a ``finite approximation of the pressure function on finite models'', conveyed by \Cref{lem:opt_ratio}. To handle a certain technical condition regarding periodicity, we use a trick which involves taking a direct product of the original subshift with a full $\Gamma$-shift (\Cref{lem:cheating}).

	In~\Cref{sec:convexgeneralization} we extend the main result of the previous section to functions which are not necessarily locally constant in~\Cref{thm:LR_etale_good_bis} ($=$ \Cref{thm:LR_etale_good}). Our argument uses tools from convex analysis and a characterization of étale Gibbs measures in terms of tangent functionals for the sofic pressure function. We also discuss the space of absolutely-summable interactions (\Cref{subsec:ASinteractions}) and the space of maps with absolutely-summable variation (\Cref{subsec:FSvariation}) and provide the proofs for~\Cref{thm:LR_norm_summable_interaction_bis} ($=$ \Cref{thm:LR_norm_summable_interaction}) and~\Cref{thm:LR_SV}.
	In~\Cref{sec:applications} we present in more detail the applications of our main theorem discussed above. In \cref{sec:beyond_sofic} we briefly discuss hypothetical extensions beyond sofic groups.
	\section{Preliminaries}\label{sec:preliminaries}

	\subsection{Spin systems and shift spaces}
	
	Let $\ag$ be a finite set, which we think of as ``the allowable spins'' or ``the alphabet'', and let $\Gamma$ be a countable group, which will play the both the role of ``the lattice'' and the ``spatial  symmetries''. The set $\ag^\Gamma = \{ x\colon \Gamma \to \ag\}$ is endowed with the product topology, where $\ag$ is considered as a discrete topological space. Elements of $\ag^{\Gamma}$ are called \define{configurations}. Given a configuration $x \in \ag^{\Gamma}$ we shall denote its value at $g \in \Gamma$ either by $x(g)$ or $x_g$.
	
	A \define{pattern} with support $F\Subset \Gamma$ is an element $p \in \ag^F$ (we use the notation $A \Subset B$ to denote that $A$ is a finite subset of $B$). We denote by $[p] = \{ x \in \ag^{\Gamma} : x(g)=p(g) \mbox{ for every } g \in F   \}$ the cylinder generated by a pattern $p$.  A \define{space of configurations} is a compact subset $X\subset \ag^{\Gamma}$.
	For a space of configurations $X$ and $F \Subset \Gamma$ the set of $X$-admissible  $F$-patterns is defined as:
	\[ P_F(X) \isdef \{ x|_F :~ x \in X\}.\]
	We denote the set of all $X$-admissible patterns by $P(X) \isdef \bigcup_{F \Subset \Gamma}P_F(X)$.
	
	The (left) \define{shift} action $\Gamma \curvearrowright \ag^{\Gamma}$ is given by 
	\[ gx(h) \isdef x(g^{-1}h) \qquad \mbox{ for every } g,h \in \Gamma 
	\mbox{ and } x \in \ag^\Gamma. \]
	Then $\Gamma \curvearrowright \ag^{\Gamma}$ is  a topological dynamical system, which is known as the \define{full $\Gamma$-shift} over $\ag$. %
	
	A subset $X \subset \ag^\Gamma$ is a \define{shift space or subshift} if and only if it is $\Gamma$-invariant and closed in the product topology. Equivalently, $X$ is a subshift if and only if there exists a (possibly infinite) set of patterns $\mathcal{F}$ such that 
	\[X=\left\{ x\in \ag^\Gamma : gx \notin [p], \mbox{ for every } g \in \Gamma, p \in  \mathcal{F} \right\}.\]
	
	If $X$ satisfies the equation above, we refer to $\mathcal{F}$ as a set of \define{forbidden patterns} or \define{hard constraints} of $X$. Note that the set $\mathcal{F}$ above is not uniquely determined by $X$. A \define{shift of finite type (SFT)} is a shift space $X$ that can be defined by a \define{finite} set of forbidden patterns.

	Given two (possibly infinite) subsets $F,F' \subset \Gamma$, $p \in \ag^{F}$ and $q \in \ag^{F'}$ such that $p|_{F \cap F'} = q|_{F \cap F'}$, we define their concatenation $p \vee q$ as the map $w\colon F \cup F' \to \ag$ whose restrictions to $F$ and $F'$ coincide with $p$ and $q$ respectively.
	
	\subsection{The (\'etale) asymptotic relation}\label{sec:etale}
	
	\begin{definition}
		We say that two patterns $p,q \in \ag^{F}$ are \define{interchangeable} in a subshift $X \subset A^{\Gamma}$ if for every $x \in A^{\Gamma}$ we have
		\[ x|_{\Gamma \setminus F}\vee p \in X \mbox{ if and only if } x|_{\Gamma \setminus F} \vee q \in X.  \]
		Where $x|_{\Gamma \setminus F}$ denotes the restriction of $x$ to $\Gamma \setminus F$. In words, $p$ and $q$ are interchangeable if occurrences of $p$ can be interchanged by $q$ on any configuration without introducing any forbidden pattern and vice-versa.

		We denote the collection of interchangeable pairs of patterns with support $F$ in $\ag^F \times \ag^F$ by $\mathcal{I}_F(X)$ and by $\mathcal{I}(X) = \bigcup_{F \Subset \Gamma}\mathcal{I}_F(X)$ the collection of interchangeable pairs of patterns.
	\end{definition}
	\begin{definition}
		For a finite set $F \Subset \Gamma$ let 
		\[\cT_F(X) \isdef \left\{ (x,x') \in X \times X : x|_{\Gamma \setminus F} = x'|_{\Gamma \setminus F} \right\}\]
		\[\cT(X)  \isdef \bigcup_{F \Subset \Gamma} \cT_F(X). \]
		Then $\cT(X)$ is a Borel equivalence relation on $X$ where each equivalence class is countable, which we call \define{the asymptotic relation on $X$}. Some authors, for instance in~\cite{BorMac_2020,Meyerovitch_2013,PetersenSchmidt1997}, call $\cT(X)$ the Gibbs relation on $X$. 
		We also define:
		\[\cT_F^0(X) \isdef \left\{ (x,x') \in \cT_F(X)  :  (x|_F ,x'|_F \right) \in \mathcal{I}_F(X) \}.\]
		\[\cT^0(X) \isdef \bigcup_{F \Subset \Gamma} \cT_F^0(X). \]
		The equivalence relation $\cT^0(X)$ can be equivalently described as the orbit equivalence relation of the countable group of
		homeomorphisms of $X$ generated by the maps $I_{p,q} \colon X \to X$ which interchange patterns $(p,q) \in \mathcal{I}_F(X)$ for some $F\Subset \Gamma$, that is 
		\begin{equation}
			I_{p,q}(x) \isdef \begin{cases}
				q \vee x|_{\Gamma \setminus F} & \mbox{ if } x|_F = p\\
				p \vee x|_{\Gamma \setminus F} & \mbox{ if } x|_F = q\\
				x & \mbox{ otherwise}.
			\end{cases}  
		\end{equation}
		We refer to $\cT^0(X)$ as the \define{\'etale asymptotic relation}. We remark that in~\cite{Meyerovitch_2013} it was called the ``topological Gibbs relation''.
	\end{definition}
	
	We endow both $\cT(X)$ and $\cT^0(X)$ with the inductive limit topology that comes from viewing $\cT(X)$ and $\cT^0(X)$ as the direct limit of the compact topological spaces $(\cT_F(X))_{F \Subset \Gamma}$ and $(\cT^0_F(X))_{F \Subset \Gamma}$ respectively. With that topology the spaces $\cT(X)$ and $\cT^0(X)$ are metrizable and locally compact, but in general not compact. We remark that this topology on $\cT^0(X)$ turns it into an \define{\'etale equivalence relation} in the sense of~\cite{Put18}. Notice that a function $\Psi\colon \mathcal{T}^0(X) \to \RR$ is continuous if and only if for every interchangeable pair of patterns $(p,q)$ the map 
	$x \mapsto \Psi(x,I_{p,q}(x))$ is a continuous function with respect to the product topology on $X$.
	
	\begin{definition}
		Given an equivalence relation $\mathcal{R}\subset X \times X$, an $\mathcal{R}$-cocycle is a function $\Psi\colon \mathcal{R} \to \mathbb{R}$ that satisfies \[
		\Psi(x,z) = \Psi(x,y) + \Psi(y,z)  \mbox{ for every } (x,y),(y,z) \in \mathcal{R}.\]
	\end{definition}
	
	From the definition of $\mathcal{R}$-cocycle it follows immediately that $\Psi(x,x)=0$ for every $x \in X$ and that $\Psi(y,x)=-\Psi(x,y)$ for every $(x,y)\in \mathcal{R}$.

	\begin{definition}\label{defn:Psi_f}
		Given an equivalence relation $\mathcal{R}\subset X \times X$ and a function $f\colon X \to \mathbb{R}$ one can formally define for $(x,y)\in \mathcal{R}$
		\[
		\Psi_f(x,y) = \sum_{g \in \Gamma} \left(f(gy) -f(gx)\right).
		\]
	\end{definition}
	If the series above is absolutely convergent for every $(x,y)\in \mathcal{R}$ then $\Psi_f$ is an $\mathcal{R}$-cocycle and is furthermore $\Gamma$-invariant.

	A function $f\colon X \to \mathbb{R}$ is \define{locally constant} if there exists $F \Subset \Gamma$ such that $f(x)=f(y)$ whenever $x|_F = y|_F$. We also say that $f$ is $F$-locally constant if we want to make $F$ explicit. For locally constant functions there are finitely many non-zero terms in the series defining $\Psi_f$ for the relations $\cT(X)$ and $\cT^0(X)$, so the series is absolutely convergent. Furthermore, whenever $f$ is locally constant, the function $\Psi_f\colon \cT^0(X) \to \mathbb{R}$ is continuous.

	Recall that $\mathcal{T}^0(X)$ is a locally compact topological space. Thus, the space of continuous functions from  $\mathcal{T}^0(X)$ to $\mathbb{R}$ can be equipped with the topology of uniform convergence on compact sets, which induces a topology on the space of continuous $\mathcal{T}^0(X)$-cocycles. There are many metrics that generate this topology. For future reference we write a specific metric $d_{\mathcal{T}^0(X)}$ that generates the topology:
	\begin{definition}\label{def:metric_T_cocycles}
		The metric $d_{\mathcal{T}^0(X)}$ on the space of $\mathcal{T}^0(X)$-cocycles is given by:
		\begin{equation}
			d_{\mathcal{T}^0(X)}(\Phi,\Psi) = \sum_{n=1}^\infty\frac{1}{2^n}
			\min\left\{1, \sup_{(x,y) \in \mathcal{T}_{F_n}^0(X)} \left|\Phi(x,y)-\Psi(x,y) \right|\right\},
		\end{equation}
		where $(F_n)_{n \geq 1}$ is some fixed enumeration of the non-empty finite subsets of $\Gamma$. 
	\end{definition}
	This metric generates the topology of uniform convergence on compact sets on the space of continuous $\mathcal{T}^0(X)$-cocycles.
	Informally speaking, two cocycles $\Phi$ and $\Psi$ are close if there exists a large $N \in \NN$ such that $|\Psi(x,y)-\Phi(x,y)|$ is uniformly small over $(x,y) \in \mathcal{T}^0_{F_n}(X)$ for every $n \leq N$.
	
	\begin{definition}
		Given $f,g \in C(X)$ such that the $\mathcal{T}^0(X)$-cocycles $\Psi_f,\Psi_g$ are well defined and continuous, let
		\[
		\rho_{\mathcal{T}^0(X)}(f,g)=\|f-g\|_\infty + d_{\mathcal{T}^0(X)}(\Psi_f,\Psi_g).\]
	\end{definition}
	The map $\rho_{\mathcal{T}^0(X)}$ is a metric on the space of continuous functions for which the $\mathcal{T}^0(X)$-cocycles are well defined and continuous. It says that two functions are close if they are uniformly close, and the cocycles they define are close in the topology of uniform convergence on compact sets on $\cT^0(X)$.

	\subsection{Gibbs measures} 
	In this section we formally define Gibbs measures (also called Gibbs states). In the literature it is possible to find several definitions for Gibbs measures, which are equivalent, at least when  the setup is sufficiently restricted. Here we essentially follow the ``conformality'' framework as in~\cite{DenkerUrbanski1991,PetersenSchmidt1997}. This definition is equivalent to the Dobru\v{s}in-Lanford--Ruelle definition of Gibbs measures as shown in Section 3 of~\cite{BorMac_2020}. The conformality framework can also be expressed as an invariance property with respect to certain probability kernels as in~\cite{Alpeev2016}. In some situations, for instance for the case of finite-range interactions on shift spaces satisfying the pivot property, one can also characterize Gibbs measures by their invariance under the Glauber dynamics as in~\cite{shriver2020free}. In what follows, all measures considered will be Borel probability measures.

	We first recall the notion of the Radon-Nikod\'ym cocycle for a measure which is non-singular with respect to a countable Borel equivalence relation in the sense of \cite{FeldmanMoore1977}.
	A Borel measure $\mu$ on a topological space $X$ is \define{non-singular} with respect to a countable Borel equivalence relation $\mathcal{R}\subset X \times X$ if whenever $\mu(A)=0$ for some Borel set $A \subset X$ then
	\[ \mu \left( \bigcup_{x \in A} \left\{ y \in X : (x,y)\in \mathcal{R}\right\} \right)=0.  \]

	If $\mu$ is non-singular with respect to a countable Borel equivalence relation $\mathcal{R}$, then a \define{Radon-Nikod\'ym $\mathcal{R}$-cocycle} for $\mu$ is a measurable map $\mathcal{D}_{\mu,\mathcal{R}}\colon \mathcal{R} \to \RR_+$ such that for every Borel bijection $\phi\colon X \to X$ satisfying  $(x,\phi(x)) \in \mathcal{R}$ for $\mu$-almost every $x \in X$, we have that \[ \frac{\dd \mu \circ \phi}{\dd \mu}(x) =  \mathcal{D}_{\mu,\mathcal{R}}(x,\phi(x)) \mbox{ for $\mu$-almost every } x \in X.  \]

	For any Borel Bijection $\phi \colon X\to X$ the Radon-Nikod\'ym derivative $\frac{\dd \mu \circ \phi}{\dd \mu}$ is uniquely defined up to a $\mu$-null set. Thus the Radon-Nikod\'ym cocycle is ``essentially unique'' in the following sense: If $\mathcal{D}_1,\mathcal{D}_2\colon \mathcal{R} \to \RR_+$ are Radon-Nikod\'ym $\mathcal{R}$-cocycles of $\mu$ then there exists a Borel set $X' \subset X$ such that $\mu(X \setminus X')=0$ and  $\mathcal{D}_1$ and $\mathcal{D}_2$ coincide on $(X' \times X) \cap \mathcal{R}$.

	In what follows, we will say that two $\mathcal{R}$-cocycles coincide up to a $\mu$-null set if there exists a Borel set $X' \subset X$ with $\mu(X \setminus X')=0$ such that the cocycles coincide on $(X' \times X) \cap \mathcal{R}$. For $x \in X$ and $F \Subset X$ we denote
	\[ [x]_F \isdef [x|_F] = \left\{y\in X : y|_F = x|_F \right\}. \]
	
	Recall that the countable Borel equivalence relation $\cT^0(X)$ is generated by the homeomorphisms $I_{p,q}$ defined in~\Cref{sec:etale}. See also~\cite[Lemma 2.2]{Meyerovitch_2013}. 
	As explained in \cite{FeldmanMoore1977} the Radon-Nikod\'ym cocycle of a non-singular measure on a countable Borel equivalence relation is ``essentially uniquely determined'' by the Radon-Nikod\'ym derivatives on a set of Borel bijections that generate the relation.
	Using this fact it is easy to obtain an explicit formula for a representative of the Radon-Nikod\'ym cocycle of $\cT^0(X)$:
	\begin{lemma}\label{lemma:RN-cocycle-form-MCT}
		Let $\mu$ be a measure on $X$ which is non-singular with respect to $\cT^0(X)$. Then there exists $X' \subset X$ such that $\mu(X \setminus X')=0$ and for every $(x,y) \in (X' \times X) \cap \cT^{0}(X)$ we have \[ \mathcal{D}_{\mu,\cT^0(X)}(x,y) = \lim_{F \nearrow \Gamma} \frac{\mu([y]_F)}{\mu([x]_F)}. \]
	\end{lemma}
	
	\begin{proof}
		Let $(x,y)\in \cT^0(X)$, then for large enough $F \Subset \Gamma$ the involution $I_{x|_F,y|_F}$ satisfies that \[ \mu([y]_F) =(\mu \circ I_{x|_F,y|_F})([x]_F) =  \int_{[x]_F}\mathcal{D}_{\mu,\cT^0(X)}(z,I_{x|_F,y|_F}(z))\dd\mu(z). \]

		Consider the filtration $(\sigma(A^{F}))_{F \Subset \Gamma}$ where the sets are ordered by inclusion and $\sigma(A^F)$ is the $\sigma$-algebra generated by $A^F$. The minimal $\sigma$-algebra generated by this filtration is the Borel $\sigma$-algebra of $A^{\Gamma}$. The martingale convergence theorem implies that for any $f \in L^1(A^{\Gamma})$ the net of conditional expectations $(\mathbb{E}_{\mu}(f|\sigma(A^F)))_{F \Subset \Gamma}$ converges to $f$ as $F \nearrow \Gamma$ both in $L^1(A^{\Gamma})$ and $\mu$-almost surely. By the Radon-Nikod\'ym theorem we have that the map $z \mapsto \mathcal{D}_{\mu,\cT^0(X)}(z,I_{x|_F,y|_F}(z))$ is in $L^1(A^{\Gamma})$. Notice that for a fixed $F\Subset \Gamma$ and $\mu$-almost every $x\in A^{\Gamma}$ \[ \mathbb{E}_{\mu}(\mathcal{D}_{\mu,\cT^0(X)}(\cdot,I_{x|_F,y|_F}(\cdot))| \sigma(A^F) ) (x)= \frac{1}{\mu([x]_F)}\int_{[x]_F}\mathcal{D}_{\mu,\cT^0(X)}(z,I_{x|_F,y|_F}(z))\dd\mu(z).\]
		
		Letting $F$ be large enough we get that $\mathcal{D}_{\mu,\cT^0(X)}(x,y)$ and $\frac{\mu([y]_F)}{\mu([x]_F)}$ are arbitrarily close outside a set of measure $0$.
	\end{proof}

	\begin{definition}
		Let $f\colon X \to \mathbb{R}$ be such that the series defining $\Psi_f(x,y)$ is absolutely convergent for every $(x,y)\in \mathcal{T}(X)$. A Borel probability measure $\mu$ on a subshift $X\subset A^{\Gamma}$ is 
		\begin{enumerate}
			\item \define{Gibbs} with respect to $f$, if $\mu$ is non-singular with respect to $\cT(X)$ and its Radon-Nikod\'ym $\cT(X)$-cocycle is equal to  $\exp(\Psi_f)$ up to a $\mu$-null set.
			\item \define{\'Etale Gibbs} with respect to $f$, if $\mu$ is non-singular with respect to $\cT^0(X)$ and its Radon-Nikod\'ym $\cT^0(X)$-cocycle is equal to $\exp(\Psi_f)$ up to a $\mu$-null set.
		\end{enumerate}
		
	\end{definition}

	As $\cT^0(X)$ is a subrelation of $\cT(X)$, it follows that for any potential $f\colon X \to \mathbb{R}$ any Gibbs measure with respect to $f$ is automatically étale Gibbs with respect to $f$. The following example shows that the converse does not hold in every subshift.
	
	\begin{example}
		Let \[ X_{\leq 1} = \{ x \in \{  \symb{0},\symb{1}  \}^{\Gamma} : |\{ g \in \Gamma : x(g)=\symb{1}   \}| \leq 1 \}  \]
		be the \define{sunny-side up} subshift. It is clear that for an infinite group $\Gamma$, the only $\Gamma$-invariant measure on $X_{\leq 1}$ is the delta measure supported on the constant configuration $0^{\Gamma}$. This measure is singular with respect to $\cT(X_{\leq 1})$, hence not Gibbs for any interaction. However, this measure is \'etale Gibbs (with respect to any $f$), since the \'etale asymptotic relation is given by $\cT^0(X_{\leq 1}) = \{(0^{\Gamma},0^{\Gamma})\} \cup \{(x,y) : x,y \in X_{\leq 1} \setminus \{0^{\Gamma}\} \}$.
	\end{example}
	
	\subsection{The topological Markov property}\label{subsec:TMP}
	
	In what follows, we shall present the topological property which characterizes the equality of the asymptotic relation $\cT(X)$ and the étale asymptotic relation $\cT^0(X)$, and therefore ensures that both definitions of Gibbs given above coincide.
	
	\begin{definition}
		We say a subshift $X \subset A^{\Gamma}$ satisfies the \define{topological Markov property} (TMP) if for every $A \Subset \Gamma$ there exists $B \Subset \Gamma$ with $A\subset B$ such that whenever $x,y$ satisfy that $x|_{B\setminus A} = y|_{B\setminus A}$ we have that the configuration $z \in A^{\Gamma}$ is in $X$, where $z(g)$ is given for every $g \in \Gamma$ by \[ z(g) = \begin{cases}
			x(g) & \mbox{ if } g \in B\\
			y(g) & \mbox{ if } g \in \Gamma \setminus B.\\
		\end{cases}  \]
		Any set $B$ satisfying the above is called a \define{memory set} of $A$.
	\end{definition}
	
	The topological Markov property, as we present it here, was defined in~\cite{BGMT_2020}. It was proposed as a generalization of the condition of being a subshift of finite type which was sufficient to prove a generalization of the Lanford--Ruelle theorem for actions of amenable groups. The TMP in turn generalizes the more restrictive notion of Topological Markov field, which was explored by several authors~\cite{ChanThesis,chandgotia2014,ChanMey2016} as a generalization of subshifts of finite type in the context of Gibbs theory. The related notion of ``splicable metric space'', which can be interpreted in the symbolic setting as a bounded variant of the TMP, was used even earlier by Gromov to provide analogues of Ax's surjunctivity theorem beyond algebraic varieties (see~\cite[Section 8.C']{Gromov1999}, and also~\cite{CeccCoorLi2021} for an application of the topological Markov property to surjunctivity of general group actions). More generally, an analogous version of the TMP can be defined for arbitrary group actions on compact metrizable spaces, and it naturally generalizes the well known pseudo-orbit tracing property, also called shadowing. See~\cite{Barbieri_Ramos_Li_2022} for further background.
	
	The class of subshifts that satisfy the topological Markov property is much larger than the class of shifts of finite type. For instance, it is known that for a fixed group $\Gamma$ there are countably many subshifts of finite type up to topological conjugacy, whereas for $\Gamma =\ZZ^2$ there exist uncountably many non-conjugate subshifts with the topological Markov property~\cite[page 233]{chandgotia2014}. Moreover, every subshift which has a trivial asymptotic relation satisfies the property~\cite[Proposition 5.3]{Barbieri_Ramos_Li_2022}. Another interesting family of examples is algebraic: every subshift whose alphabet is a finite group and is closed under the pointwise group operation satisfies the property~\cite[Proposition 5.1]{BGMT_2020}, while there are examples with that structure which are not of finite type if the acting group is solvable but not polycyclic-by-finite~\cite{Salo2018_groupshiftSFT}.
	
	It is clear from the definition that if $X$ satisfies the TMP and $B$ is a memory set of $A$, then all patterns with support $B$ which coincide on $B\setminus A$ and occur in some configuration of $X$ are pairwise interchangeable. This property in fact characterizes the spaces where the étale asymptotic relation coincides with the asymptotic relation as shown in the next proposition.
	
	\begin{proposition}\label{prop:GibbsTMP}
		Let $X\subset A^{\Gamma}$ be a subshift. Then $X$ satisfies the TMP if and only if $\cT^0(X)= \cT(X)$. 
		
		In particular, if $X$ satisfies the TMP and $f\colon X \to \mathbb{R}$ is such that the series defining $\Psi_f(x,y)$ is absolutely convergent for every $(x,y)\in \mathcal{T}(X)$, then every Borel probability measure is Gibbs with respect $f$ if and only if it is \'etale Gibbs with respect to $f$. 
	\end{proposition}
	
	\begin{proof}
		Suppose $X$ satisfies the TMP. Obviously, it is always true that $\cT^0(X) \subset \cT(X)$. Let $(x,y)\in \cT(X)$. Then there exists $F \Subset \Gamma$ such that $(x,y) \in \cT_F(X)$. As $X$ satisfies the TMP, there exists a memory set $B$ for $F$, and therefore the pair of words $x|_{B}$ and $y|_{B}$ are interchangeable, hence $(x,y) \in \cT^0_B(X)$.
		
		Conversely, suppose $X$ does not satisfy the TMP and let $(F_n)_{n \in \NN}$ be an increasing sequence of finite subsets of $\Gamma$ such that $F_n \nearrow \Gamma$. Then there is $A\Subset \Gamma$ such that for every $F_n$ which contains $A$, there are $x^{(n)},y^{(n)}\in X$ such that $x^{(n)}|_{F_n \setminus A} = y^{(n)}|_{F_n \setminus A}$ but the patterns $p_n = x^{(n)}|_{F_n}$ and $q_n = y^{(n)}|_{F_n}$ are not interchangeable. Let $(x',y')$ be an accumulation point of the sequence $(x^{(n)},y^{(n)})_{n \in \NN}$. By definition, $(x',y')\in \cT_A(X)$ and thus $(x',y')\in \cT(X)$.
		
		It follows that for each $K \Subset \Gamma$ which contains $A$ there is $m \in \NN$ such that $p_m|_K = x'|_K$ and $q_m|_K= y'|_K$. Thus $(x',y')\notin \cT_K^0(X)$ for any large enough $K$ which implies that $(x',y')\notin \cT^0(X)$. This shows that $\cT^0(X)\neq \cT(X)$.
	\end{proof}
	
	\section{Sofic groups and sofic pressure}\label{sec:soficgroupspressure}
	
	\subsection{Sofic groups}
	
	For a finite set $V$ we write $\Sym(V)$ for the group of permutations of $V$. A group $\Gamma$ is \define{sofic} if there exist a sequence $(V_{i})_{i\in \NN}$ of finite sets such that $|V_i|=n_i$ goes to infinity and a sequence $\Sigma=\{\sigma_i \colon \Gamma\rightarrow \Sym(V_i) \}_{i=1}^{\infty}$ that is
	\begin{align}
		\mbox{asymptotically an action: } & \lim_{i\rightarrow\infty} \frac{1}{n_i} \left| \left\{ v\in V_{i}
		: \sigma_{i}(st)v=\sigma_{i}(s)\sigma_{i}(t)v\right\}\right|   =1&
		\mbox{ for every } s,t\in \Gamma\\
		\mbox{asymptotically free: } & \lim_{i\rightarrow\infty} \frac{1}{n_i}\left|\left\{ v\in V_{i}
		: \sigma_{i}(s)v\neq\sigma_{i}(t)v\right\}\right| =1  &  \mbox{ for every } s\neq t\in \Gamma.
	\end{align}
	In this case we say $\Sigma$ is a \define{sofic approximation sequence} of $\Gamma$.
	
	\begin{definition}
		Given  $\sigma\colon \Gamma \to \Sym(V)$  and $F \Subset \Gamma$, we say that $v \in V$ is \define{$F$-good for $\sigma$} if $\sigma(st)v=\sigma(s)\sigma(t)v$ for all $s,t \in F$ and $\sigma(s)v \ne \sigma(t)v$ for all $s \ne t \in F$.
	\end{definition}
	With this notation, saying that $\Sigma=\{\sigma_i \colon \Gamma\rightarrow \Sym(V_i) \}_{i=1}^{\infty}$ is a sofic approximation means that for any $F \Subset \Gamma$,
	\[
	\lim_{i \to \infty}\frac{1}{n_i}\left|\left\{ v\in V_{i}
	: v \mbox{ is } F\mbox{-good for }\sigma_i \right\}\right| =1.
	\]
	\subsection{Pullback names and empirical distributions}
	For the remainder of this section, we fix a compact metrizable space $X$, an action $\Gamma \curvearrowright X$. we denote by $\Prob(X)$ the space of Borel probability measures on $X$, and by $\Prob_{\Gamma}(X)$ the space of $\Gamma$-invariant Borel probability measures on $X$. We also consider the space $X^{\Gamma}$ with the product topology and the natural left action $\Gamma \curvearrowright X^{\Gamma}$ given by $(gx)(h) = x(g^{-1}h)$ for every $g,h \in \Gamma$ and $x \in X^{\Gamma}$.
	
	The following notation closely follows Austin's approach to sofic entropy \cite{Austin2016Additivity}:
	\begin{definition}
		Let $X$ be a compact metrizable space, $V$ be a finite set, $v \in V$ and $\sigma \colon \Gamma \to \Sym(V)$. We define the following maps
		\begin{enumerate}
			\item $\xi_{\sigma,v} \colon X^{V} \to X^{\Gamma}$ given by \[ \xi_{\sigma,v}(x)(g) = x(\sigma(g^{-1})v) \mbox{ for every } g \in \Gamma.  \]
			\item $\xi_{\sigma}\colon X^V \to \mbox{Prob}(X^{\Gamma})$ given by
			\[ \xi_{\sigma}(x) = \frac{1}{|V|}\sum_{v \in V}{\delta_{\xi_{\sigma,v}(x)}}.  \]
			\item $\zeta_{\sigma} \colon \mbox{Prob}(X^{V}) \to \mbox{Prob}(X^{\Gamma})$ given by \[ \zeta_{\sigma}(\nu)(A) = \int_{X^{V}} \xi_{\sigma}(x)(A) \dd \nu (x) \mbox{ for every Borel set }A.    \]	
		\end{enumerate}
	\end{definition}
	In \cite{Austin2016Additivity} $\xi_{\sigma,v}(x)$ was called ``the \define{pullback name} of $x$ via $\sigma$ at $v$'', and $\xi_\sigma(x)$ was called  ``the \define{empirical distribution} of $x$''. 
	
	In the case where $X$ is a discrete finite set (i.e, an alphabet), for $\nu \in \mbox{Prob}(X^{V})$ we can write $\nu = \sum_{ x \in X^V}\lambda_x \delta_{x}$ and thus we may express the map $\zeta_{\sigma}$ in the following convenient way \[ \zeta_{\sigma}(\nu) = \sum_{x \in X^V}{\lambda_x \xi_{\sigma}(x)} = \sum_{x \in X^V} \frac{\lambda_x}{|V|}\sum_{v \in V}{\delta_{\xi_{\sigma,v}(x)}}.   \]
	
	Recall that a basis for $X^{\Gamma}$ is given by the \define{cylinder sets} i.e. sets $U \subset X^\Gamma$ of the form  \[U = \{ x \in X^{\Gamma}: x(s) \in U_s \mbox{ for every }s \in S\},\]  where $S\Subset \Gamma$ and $U_s \subset X$ is open for every $s \in S$.
	
	\begin{proposition}\label{prop:sofic_limit_measure_invariant}
		Let $\Sigma=\{\sigma_i \colon \Gamma\rightarrow \Sym(V_i) \}_{i=1}^{\infty}$ be a sofic approximation sequence of $\Gamma$. For any $\varepsilon>0$, $g \in \Gamma$ and cylinder set $U\subset X^{\Gamma}$, there is $N \in \NN$ such that for any $i \geq N$ and $\nu_i \in \mbox{Prob}(X^{V_i})$ we have $|\left(\zeta_{\sigma_i}(\nu_i)- g\zeta_{\sigma_i}(\nu_i) \right)(U)|\leq \varepsilon$. In particular, Any weak-$*$ limit of a sequence of measures $(\zeta_{\sigma_i}(\nu_i))_{i \in \NN}$ with $\nu_i \in \mbox{Prob}(X^{V_i})$ is a $\Gamma$-invariant probability measure on $X^{\Gamma}$.
	\end{proposition}
	
	\begin{proof}
		Let $\varepsilon>0$ be arbitrary, $S\Subset \Gamma$ be a symmetric set containing the identity,  and let $U = \{ x \in X^{\Gamma}: x(s) \in U_s \mbox{ for every }s \in S\}$, with $U_s \subset X$  open in $X$ for every $s \in S$.
		Let $g \in \Gamma$. Then for $\nu_i \in \mbox{Prob}(X^{V_i})$ we may write
		
		\begin{align}
			\left(\zeta_{\sigma_i}(\nu_i)- g\zeta_{\sigma_i}(\nu_i) \right)(U) &  = \int_{X^{V_i}} \frac{1}{|V_i|}\sum_{v \in V_i}{\left(\delta_{\xi_{\sigma_i,v}(x)}(U) - \delta_{\xi_{\sigma_i,v}(x)}(g^{-1}U)\right)} \dd \nu_i.\\
			&  = \int_{X^{V_i}}  \frac{1}{|V_i|} \left|\{v \in V_i : \xi_{\sigma_i,v}(x)(s) \in U_s \mbox{ for every } s \in S \}\right| \dd \nu_i \\
			&  \quad - \int_{X^{V_i}}  \frac{1}{|V_i|}\left| \{v \in V_i : \xi_{\sigma_i,v}(x)(g^{-1}s) \in U_s \mbox{ for every } s \in S \}\right| \dd \nu_i.
		\end{align}
		
		Notice that $\xi_{\sigma_i,v}(x)(s) = x(\sigma_i(s^{-1})(v))$ and that $\xi_{\sigma_i,v}(x)(g^{-1}s) = x(\sigma_i(s^{-1}g)(v))$. As $\Sigma$ is asymptotically an action, for every large enough $i\in \NN$ there exists $V_i'' \subset V_i$ which is $(S\cup gS \cup g^{-1} S)$-good for $\sigma_i$ and such that $|V_i''|\geq (1-\frac{\varepsilon}{4})|V_i|$. Let $V_i'  = V_i'' \cap \sigma_i(g^{-1})V_i''$. Then $|V_i'| \geq (1-\frac{\varepsilon}{2})|V_i|$.
		We get that  that $\xi_{\sigma_i,v}(x) \in U$ if and only if $\xi_{\sigma_i,\sigma_i(g^{-1})v}(x)\in g^{-1}U$ for every $v \in V'_i$. We can thus eliminate vertices in $V_i'$ in the equation above and obtain the following estimate:
		
		\begin{align}
			\left\vert \left(\zeta_{\sigma_i}(\nu_i)- g\zeta_{\sigma_i}(\nu_i) \right)(U) \right\vert \leq  \int_{X^{V_i}} \frac{2|V_i \setminus V'_i| }{|V_i|} \dd \nu_i \leq \varepsilon.
		\end{align}
		
		As the cylinders sets are a basis for the product topology on $X^{\Gamma}$ the result follows.\end{proof}
	
	When $\Gamma$ is countable, both the product topology of $X^{\Gamma}$ and the space $\Prob(X^{\Gamma})$ are metrizable. Let us denote a compatible metric on $X^{\Gamma}$ by $d_{T}$ and a compatible metric on $\Prob(X^{\Gamma})$ by $d_{P}$ (for instance, the Kantorovich metric). The previous lemma can then be restated in the following way: for every $\varepsilon>0$ and $g \in \Gamma$ there is $N \in \NN$ such that for every $i \geq N$ we have ${d}_{P}(\zeta_{\sigma_i}(\nu_i),g \zeta_{\sigma_i}(\nu_i) ) \leq \varepsilon$ for any measure $\nu_i$ on $X^{V_i}$.

	Given a closed $\Gamma$-invariant set $Y \subset X^\Gamma$, $\mu  \in \Prob(X^{\Gamma})$ and $\delta >0$, we define:
	\[N_{\delta}(Y) \isdef \{  \nu \in \Prob(X^{\Gamma}) : \nu(\{ x \in X^{\Gamma} : \min_{y \in Y} d_{T}(x,y)\geq\delta  \})\leq\delta   \}, \]
	\[ N_{\delta}(\mu) = \{  \nu \in \Prob(X^{\Gamma}) : {d}_{P}(\mu,\nu)\leq\delta   \}. \]
	
	\begin{lemma}\label{lem:sofic_measures_accum}
		Let $\Gamma$ be a countable group, $\Sigma=\{\sigma_i \colon \Gamma\rightarrow \Sym(V_i) \}_{i=1}^{\infty}$ be a sofic approximation sequence of $\Gamma$ and $Y\subset X^{\Gamma}$ a closed $\Gamma$-invariant set. For any $\delta >0$ there exists $\delta' >0$ and $N\in \NN$ such that for any $i \geq N$,
		\[
		\xi_{\sigma_i}^{-1}({N_{\delta'}(Y)}) \subset \bigcup_{\mu \in \Prob_\Gamma(Y)}\xi_{\sigma_i}^{-1}(N_{\delta}(\mu)).
		\]
		Furthermore, if $\Prob_{\Gamma}(Y)\neq \varnothing$, then for any $\delta >0$ there exists $\delta' >0$, $N \in \NN$ and $\mu_1,\ldots, \mu_k \in \Prob_\Gamma(Y)$ such that for any $i \geq N$,
		\[
		\xi_{\sigma_i}^{-1}({N_{\delta'}(Y)}) \subset \bigcup_{j=1}^k\xi_{\sigma_i}^{-1}(N_{\delta}(\mu_j)).
		\]
	\end{lemma}
	
	\begin{proof}
		By \Cref{prop:sofic_limit_measure_invariant},
		for any $\varepsilon >0$ and $g \in \Gamma$ there exists $N \in \mathbb{N}$ such that for every $i \geq N$ and $x \in X^{V_i}$ , 
		${d}_{P}(\zeta_{\sigma_i}(\delta_x),g \zeta_{\sigma_i}(\delta_x) ) = {d}_{P}(\xi_{\sigma_i}(x),g \xi_{\sigma_i}(x) ) \leq \varepsilon$.
		For $\delta >0$ and $g \in \Gamma$ let
		\[P(\delta,g) \isdef \{  \nu \in \Prob(X^{\Gamma}) : {d}_{P}(\nu,g\nu)\leq\delta   \}.\]
		It follows that for any $\delta'>0$ and $F\Subset \Gamma$ there exists $N\in \NN$ such that for any $i \geq N$
		\[
		\xi_{\sigma_i}(A^{V_i}) \subset \bigcap_{g \in F}P(\delta',g). 
		\]
		Now, it is clear from the definition that we have
		\[
		\Prob_\Gamma(Y) = \bigcap_{\delta' >0, F \Subset \Gamma} \left({N_{\delta'}(Y)} \cap \bigcap_{g \in F}P(\delta',g)\right).
		\]
		Since $\Prob_\Gamma(Y)  \subset \bigcup_{\mu \in \Prob_\Gamma(Y)}N_\delta(\mu)$ it follows by compactness that for every $\delta >0$  there exists
		$\delta' >0$ and 	$F \Subset \Gamma$ such that
		\[{N_{\delta'}(Y)} \cap \bigcap_{g \in F}P(\delta',g) \subset \bigcup_{\mu \in \Prob_\Gamma(Y)}N_\delta(\mu).\]
		
		As $\xi_{\sigma_i}^{-1}(P(\delta',g))=A^{V_i}$ for every $g \in F$ and $i \geq N$, it follows that 
		\[
		\xi_{\sigma_i}^{-1}\left({N_{\delta'}(Y)} \right)\subset \xi_{\sigma_i}^{-1}\left(\bigcup_{\mu \in \Prob_{\Gamma}(Y)} N_{\delta}(\mu)\right) = \bigcup_{\mu \in \Prob_{\Gamma}(Y)} \xi_{\sigma_i}^{-1}(N_{\delta}(\mu)).
		\]
		
		Since ${N_{\delta'}(X)} \cap \bigcap_{g \in F}P(\delta',g)$ is  compact, and $N_\delta(\mu)$ is open for every $\mu \in \Prob(X)$, it follows that if $\Prob_{\Gamma}(Y)$ is nonempty, then there exists $\mu_1,\ldots, \mu_k \in \Prob_\Gamma(Y)$ such that
		\[{N_{\delta'}(Y)} \cap \bigcap_{g \in F}P(\delta',g) \subset \bigcup_{j=1}^kN_\delta(\mu_j).\]
		Therefore we may conclude that\[
		\xi_{\sigma_i}^{-1}\left({N_{\delta'}(Y)} \right)\subset  \bigcup_{j=1}^k \xi_{\sigma_i}^{-1}(N_{\delta}(\mu_j)).
		\]
	\end{proof}
	
	\subsection{Topological and measure-theoretic pressure of actions of sofic groups}
	
	The notions of topological and measure-theoretic pressure for $\mathbb{Z}^d$-actions were introduced and studied by Ruelle~\cite{RuelleTAMS1973}, who also proved a variational principle extending the variational principle for entropy. Both notions and the variational principle were latter extended to actions of countable amenable groups by Ollagnier and Pinchon~\cite{ollagnier1982variational,Ollagnier1985book}, Stepin and Tagi-Zade~\cite{stepin1980variational} and Tempelman~\cite{tempelman1984specific}.
	A further generalization of the notions of measure-theoretic and topological pressure for actions of sofic groups on compact metrizable spaces was introduced by Nhan-Phu Chung in~\cite{chung_2013}, who also proved the corresponding variational principle in this setting (building on the landmark papers of Bowen \cite{Bowen2010_2} and Kerr and Li~\cite{KerrLi2011}). A similar notion is also introduced by Alpeev in~\cite{Alpeev2016} for the space $A^{\Gamma}$, albeit with different objectives.
	
	In what follows we will recall Chung's notions of pressure for actions of sofic groups, and then restate them in an equivalent manner using Austin's framework of ``empirical measures'', which we introduced above. This formalism will turn out to be convenient in our symbolic setting.

	Let $\Gamma$ be a sofic group acting on a compact metrizable space $X$ by homeomorphisms, $\Sigma= \{ \sigma_i \colon \Gamma \to \Sym(V_i)\}$ a sofic approximation sequence for $\Gamma$, $\rho$ a continuous pseudometric on $X$ and $f \colon X \to \RR$ a continuous map.
	
	Let $i \in \NN$, $F \Subset \Gamma$ and $\delta >0$. The set $\Map(\rho, F,\delta,\sigma_i)$ consists of all maps $\varphi \colon V_i \to X$ such that \[ \max_{ s \in F } \left( \frac{1}{|V_i|} \sum_{ v \in V_i} (\rho( s\varphi(v)  , \varphi(\sigma_s(v))  ))^2    \right)^{1/2} \leq \delta. \]
	
	That is, the maps $\varphi$ which roughly look like the restriction of an orbit to $F$ for most vertices $v \in V$. For $\varepsilon>0$, we say $E \subset \Map(\rho, F,\delta,\sigma_i)$ is $\varepsilon$-separated if for every $\varphi_1,\varphi_2 \in E$ we have \[ \max_{v \in V_i} \rho(\varphi_1(v),\varphi_2(v)) \geq \varepsilon.  \]
	
	Let $M^{\varepsilon}_{\Sigma,\infty}(f,X,\Gamma,\rho,F,\delta,\sigma_i)$ denote the supremum over all $\varepsilon$-separated subsets $E$ of $\Map(\rho, F,\delta,\sigma_i)$ of the expression 
	\begin{equation}\label{eq:exp_pressure_formula_basic}
		\sum_{\varphi \in E} \exp\left(\sum_{ v \in V_i} f(\varphi(v))  \right).  
	\end{equation}
	
	Finally, let \[  P^{\varepsilon}_{\Sigma,\infty}(f,X,\Gamma,\rho,F,\delta,\sigma_i) = \frac{1}{|V_i|}\log\left(M^{\varepsilon}_{\Sigma,\infty}(f,X,\Gamma,\rho,F,\delta,\sigma_i)\right) \]
	
	\begin{definition}
		The \define{topological sofic pressure} of $\Gamma \curvearrowright X$ with respect to $f \colon X \to \RR$ and the sofic approximation sequence $\Sigma$ is given by
		\[ P_{\Sigma}(\Gamma \curvearrowright X,f) = \sup_{\varepsilon >0} \inf_{F\Subset \Gamma}\inf_{\delta>0} \limsup_{i \to \infty} P^{\varepsilon}_{\Sigma,\infty}(f,X,\Gamma,\rho,F,\delta,\sigma_i).    \]
	\end{definition}
	The value of the sofic pressure function at  the zero function $f =0 $ is called the \define{topological sofic entropy} of $\Gamma \curvearrowright X$ with respect to  the sofic approximation sequence $\Sigma$. We denote the topological sofic entropy by
	\[ h_{\Sigma}(\Gamma \curvearrowright X) = P_{\Sigma}(\Gamma \curvearrowright X,0).  \]
	The topological sofic entropy was defined by Kerr and Li~\cite{KerrLi2011} before the the definition of sofic pressure, which can be considered as a generalization of their notion.
	
	If $\mu$ is a $\Gamma$-equivariant measure on $X$ and $L$ is a finite subset of $C(X)$, we define $\Map_{\mu}(\rho, F,\delta,\sigma_i,L)$ as the set of $\varphi \in \Map(\rho, F,\delta,\sigma_i)$ which satisfy \[ \left\vert \frac{1}{|V_i|}\sum_{ v \in V_i} h(\varphi(v)) - \int_X h \dd\mu \right\vert \leq \delta, \mbox{ for every }h \in L. \]
	
	Similarly, we let $M^{\varepsilon}_{\Sigma,\infty,\mu}(f,X,\Gamma,\rho,F,\delta,\sigma_i,L)$ denote the supremum over all $\varepsilon$-separated subsets of $\Map_{\mu}(\rho, F,\delta,\sigma_i,L)$ of the same expression as above, and let \[  P^{\varepsilon}_{\Sigma,\infty,\mu}(f,X,\Gamma,\rho,F,\delta,L,\sigma_i) = \frac{1}{|V_i|}\log\left(M^{\varepsilon}_{\Sigma,\infty,\mu}(f,X,\Gamma,\rho,F,\delta,L,\sigma_i)\right). \]
	\begin{definition}
		The \define{measure-theoretic sofic pressure} of $\Gamma \curvearrowright (X,\mu)$ with respect to $f \colon X \to \RR$ and the sofic approximation sequence $\Sigma$ is given by
		\[ P_{\Sigma}(f,\Gamma \curvearrowright (X,\mu)) = \sup_{\varepsilon >0} \inf_{F\Subset \Gamma}\inf_{\delta>0}\inf_{L \Subset C(X)} \limsup_{i \to \infty} P^{\varepsilon}_{\Sigma,\infty,\mu}(f,X,\Gamma,\rho,F,\delta,L,\sigma_i).    \]
	\end{definition}
	The \define{measure-theoretic sofic entropy} of $\Gamma \curvearrowright (X,\mu)$ with respect to  the sofic approximation sequence $\Sigma$ is the measure-theoretic sofic pressure of $\Gamma \curvearrowright (X,\mu)$ at the zero function $f=0$. We denote the measure-theoretic sofic entropy by
	\[ h_{\Sigma}(\Gamma \curvearrowright X,\mu) = P_{\Sigma}(\Gamma \curvearrowright X,\mu,0).  \]
	
	A straightforward approximation argument shows that 
	\[ P_{\Sigma}(\Gamma \curvearrowright X,\mu,f) = h_{\Sigma}(\Gamma \curvearrowright X,\mu) + \int_X f \dd \mu.  \]
	It is consistent with the definition to declare that for $\mu \in \Prob(X)$ which is not $\Gamma$-invariant,
	\[ P_\Sigma(\Gamma \curvearrowright X,\mu,f) = -\infty.\] 
	
	It can sometimes be semantically and conceptually convenient to replace the parameters $F \Subset \Gamma$, $\delta >0$ and $L \Subset C(X)$ that appear in the above definition of topological and measure-theoretic sofic pressure by a single parameter. For $\sigma \colon \Gamma \to \Sym(V)$, $f \colon X\to \RR$, $\varepsilon>0$ and $U$ an open set of probability measures on $X^{\Gamma}$ we define \[ P(U,\sigma,\varepsilon,f) = \frac{1}{|V|} \log \left( \sup_{E \in \mathcal{E}(U,\sigma,\varepsilon)} \sum_{\varphi \in E }\exp\left(\sum_{v \in V}f(\varphi(v))   \right)     \right), \]
	where $\mathcal{E}(U,\sigma,\varepsilon)$ is the set of all collections $E$ of maps $\varphi \in X^{V_i}$ which satisfy that $\xi_{\sigma}(\varphi)\in U$ and which are $\varepsilon$-separated.
	
	Let $X^{\star} \subset X^{\Gamma}$ be the space of $\Gamma$-orbits, that is \[ X^{\star} = \{ y \in X^{\Gamma} : y_{g^{-1}} = (gy)_{1_{\Gamma}} \mbox{ for every } g \in \Gamma    \}.   \] 
	Let $\mu^{\star}$ be the pushforward measure on $X^{\star}$ induced by a measure $\mu$ on $X$ with respect to the map that assigns to each $x \in X$ its orbit under $\Gamma$. It is known that the definitions of sofic pressure (both topological and measure-theoretic) do not depend upon the choice of continuous pseudometric, as long as it is dynamically generating, see~\cite[Lemma 2.7]{chung_2013}, from this fact, a computation shows that we may rephrase the definitions above succinctly as
	
	\[ P_{\Sigma}(\Gamma \curvearrowright X,f) = \sup_{\varepsilon >0} P_{\Sigma}(\varepsilon,\Gamma \curvearrowright X,f),     \]
	and
	\[ P_{\Sigma}(f,\Gamma \curvearrowright X,\mu,f) = \sup_{\varepsilon >0}    P_{\Sigma}(\varepsilon,\Gamma \curvearrowright X,\mu,f). \]
	Where:
	
	\[ P_{\Sigma}(\varepsilon,\Gamma \curvearrowright X,f) \isdef \inf_{\delta >0} \limsup_{i \to \infty} P(\mathcal{N}_{\delta}(X^{\star}),\sigma_i,\varepsilon,f), \]
	and
	\[ P_{\Sigma}(\varepsilon,\Gamma \curvearrowright X,\mu,f) \isdef  \inf_{\delta >0} \limsup_{i \to \infty} P(\mathcal{N}_{\delta}(\mu^{\star}),\sigma_i,\varepsilon,f).    \]
	The above equivalent expressions have a number of advantages which will become evident in the symbolic setting. We state following simple observation (essentially  equivalent to Proposition $2.2$ of~\cite{chungZhang205weakExpansiveness}):
	\begin{proposition}\label{prop:exp_upper_semi_cont_enropy}
		For any $\varepsilon >0$ the function $H_{\varepsilon}\colon\Prob_{\Gamma}(X) \to \{-\infty\}\cup [0,+\infty)$ given by
		\[H_{\varepsilon}(\mu)= h_{\Sigma}(\varepsilon,\Gamma \curvearrowright X,\mu)\]
		is upper semi-continuous.
		
	\end{proposition}
	\begin{proof}
		Fix $\varepsilon >0$ and $\mu \in \Prob_\Gamma(X)$. Choose any $y > H_{\varepsilon}(\mu)$, then there exists
		$\delta >0$ such that $\limsup_{i \to \infty}P(\mathcal{N}_{\delta}(\mu^{\star}),\sigma_i,\varepsilon,0) < y$. It follows that for any 
		$\nu \in \mathcal{N}_{\delta}(\mu^{\star})$, $ H_{\varepsilon}(\nu) < y$.
	\end{proof}
	
	With all the notation above in place, we can provide short proof of the variational principle for sofic pressure. It seems that once we have have the definitions in place, the proof below is as short as Misiurewicz's proof of the variational principle for entropy of $\mathbb{Z}^d$-actions \cite{Misiurewicz_1975}, and arguably conceptually simpler. 
	
	\begin{theorem}[N.P. Chung's variational principle for sofic pressure \cite{chung_2013}]\label{thm:variational}
		Let $\Gamma \curvearrowright X$ be an action of a sofic group $\Gamma$ on a compact space $X$ and let $\Sigma$ be a sofic approximation sequence for $\Gamma$.
		For every $\varepsilon >0$ and  $f \in C(X)$ we have:
		\[
		P_{\Sigma}(\varepsilon,\Gamma \curvearrowright X,f) = \sup_{\mu \in \Prob_{\Gamma}(X)}P_{\Sigma}(\varepsilon,f,\Gamma \curvearrowright X,\mu,f).
		\]
		In particular,
		\[
		P_{\Sigma}(\Gamma \curvearrowright X,f) = \sup_{\mu \in \Prob_{\Gamma}(X)}P_{\Sigma}(\Gamma \curvearrowright X,\mu,f).
		\]
	\end{theorem}
	
	\begin{proof}
		Suppose that $\Prob_{\Gamma}(X) \neq \varnothing$ and fix $\mu \in \Prob_{\Gamma}(X)$. Let $\delta>0$ and choose $\delta'>0$ such that whenever $d_P(\mu^{\star},\nu)<\delta'$ then for $h \in C(X^{\Gamma})$ given by  $h(x) = d_T(x,X^{\star})$ we have \[ \left|\int_{X^{\Gamma}} h \dd \mu^{{\star}} -  \int_{X^{\Gamma}} h \dd \nu \right| < \delta^2.\]
		It follows that
		\[  \delta > \frac{1}{\delta}  \left|\int_{X^{\Gamma}} h \dd \mu^{{\star}} -  \int_{X^{\Gamma}} h \dd \nu \right| = \frac{1}{\delta}\left|\int_{X^{\Gamma}} h \dd \nu \right| \geq \nu(\{ x \in X^{\Gamma} : d_{T}(x,X^{\star})\geq\delta  \}).    \]
		Thus we conclude that for each $\delta>0$, there is $\delta'>0$ such that  $N_{\delta'}(\mu^{\star}) \subset N_{\delta}(X^{\star})$. It follows that for any $i \in \mathbb{N}$, $f \in C(X)$ and $\varepsilon>0$
		\[P(N_{\delta'}(\mu^{\star}),\sigma_i, \varepsilon,f) \le P(N_{\delta}(X^{\star}),\sigma_i,\varepsilon,f).\]
		This immediately yields the ``easy direction'' of the variational principle:
		\[ \sup_{\mu \in \Prob_{\Gamma}(X)}P_{\Sigma}(\varepsilon,\Gamma \curvearrowright X,\mu,f) \le P_{\Sigma}(\varepsilon,\Gamma \curvearrowright X,f).\]
		Notice that the above inequality is trivial in the case where $\Prob_{\Gamma}(X) = \varnothing$.
		
		In order to show the converse we apply~\Cref{lem:sofic_measures_accum}. The case where $\Prob_{\Gamma}(X)=\varnothing$ yields that $\xi_{\sigma_i}^{-1}({N_{\delta'}(X^{\star})})$ is empty for all large enough $i$, and thus $P_{\Sigma}(f,\Gamma \curvearrowright X)= -\infty$ and the variational principle holds. Suppose that $\Prob_{\Gamma}(X)\neq \varnothing$, then~\Cref{lem:sofic_measures_accum} yields that for any $\delta >0$ there exists $\delta'>0$, $N \in \NN$ and $\mu_1,\ldots,\mu_k \in \Prob_\Gamma(X^{\star})$ such that for any $i >N$, 
		\[
		\xi_{\sigma_i}^{-1}({N_{\delta'}(X^{\star})}) \subset \bigcup_{j=1}^k\xi_{\sigma_i}^{-1}(N_{\delta}(\mu_j)).
		\]
		It follows that for all $i>N$, $f \in C(X)$ and $\varepsilon>0$,
		\[
		P(N_{\delta'}(X),\sigma_i,\varepsilon,f) \le \frac{\log(k)}{|V_i|}+ \max_{1\le j \le k} P(N_\delta(\mu_j),\sigma_i,\varepsilon,f).
		\]
		Let $(\delta_n)_{n=1}^\infty$ be a decreasing sequence of positive numbers tending to $0$. From the above equation we conclude that there exists another decreasing sequence of positive numbers $(\delta'_n)_{n=1}^\infty$ and a sequence of measures $\mu_n \in \Prob_\Gamma(X^\star)$ such that for every $n \in \NN$
		\[
		\limsup_{i \to \infty} P(N_{\delta'_n}(X),\sigma_i,\varepsilon,f) \le  \limsup_{i \to \infty} P(N_{\delta_n}(\mu_n),\sigma_i,\varepsilon,f).
		\]
		Let $\mu \in \Prob_\Gamma(X^\star)$ be a weak-$*$ limit of of $(\mu_n)_{n=1}^\infty$. Then for any $\delta>0$ there exists $N$ such that  for all $n >N$ we have $N_{\delta_n}(\mu_n) \subset N_\delta(\mu)$.
		It follows that for any $\delta >0$ we have 
		\[
		\limsup_{i \to \infty} P(N_{\delta'_n}(X),\sigma_i,\varepsilon,f) \le \limsup_{i \to \infty} P(N_{\delta}(\mu),\sigma_i,\varepsilon,f),
		\]
		for all sufficiently large $n$.
		Taking $n \to \infty$ and then infimum over $\delta >0$, we conclude that
		\[
		P_{\Sigma}(\varepsilon,\Gamma \curvearrowright X,f) \le P_{\Sigma}(\varepsilon,\Gamma \curvearrowright X,\mu,f).
		\]
		In particular, we obtain the variational principle for sofic pressure.
	\end{proof}
	
	\begin{definition}
		We say that $\mu \in \Prob_{\Gamma}(X)$ is an \define{equilibrium measure} (also called \define{equilibrium state}) for $f \in C(X)$ with respect to $\Sigma$ if \[P_{\Sigma}(\Gamma \curvearrowright X,f ) = P_{\Sigma}(\Gamma \curvearrowright X,\mu,f ).\]
		That is, if it achieves the supremum in the variational principle. In the case $f=0$ we call $\mu$ a \define{measure of maximal entropy} with respect to $\Sigma$.
	\end{definition}

	Let $(W,\norm{\ })$ be a Banach space and denote by $W^*$ its continuous dual space. Given a convex function $F\colon W \to \RR$, we say that a linear functional $\psi\in W^*$ is a \define{tangent functional} (or \define{subgradient}) of $F$ at $w \in W$ if for every $u \in W$ we have \[ F(u)-F(w) \geq \psi(u-w).  \]

	Given a sofic approximation sequence for $\Gamma$, and a $\Gamma$-action
	$\Gamma \curvearrowright X$, we can consider the measure-theoretic entropy and the topological pressure as functions on $\Prob(X)$ and $C(X)$:
	\[
	H \isdef h_\Sigma(\Gamma \curvearrowright X,\cdot)\colon\Prob(X) \to \{-\infty\}\cup [0,+\infty]
	\]
	and
	\[
	\Pi \isdef P_\Sigma(\Gamma \curvearrowright X , \cdot)\colon C(X) \to [-\infty,+\infty]
	\]

	The variational principle is equivalent to the statement that the topological pressure function $\Pi$ is precisely equal to the Legendre transform (or  Legendre--Fenchel transform) of $-H$.

	The following properties of the topological pressure are either classical (at least in the amenable case) or trivial, some of them can be found in~\cite[Proposition 6.1]{chung_2013}. The straightforward verification of these properties extends verbatim to the case of actions of sofic groups.  We provide short self-contained proofs for completeness.
	
	\begin{proposition}\label{prop:tangent_functionals_eq_measures}
		Let $\Sigma$ be a sofic approximation sequence for $\Gamma$, and let 
		$\Gamma \curvearrowright X$ be a $\Gamma$-action such that $h_\Sigma(\Gamma \curvearrowright X)$ is finite. 
		Then the sofic topological pressure of $\Gamma \curvearrowright X$ with respect to $\Sigma$ is finite for every $f \in C(X)$, and the function  $\Pi \colon C(X)\to \RR$ given by $\Pi(f) = P_{\Sigma}(f,\Gamma\curvearrowright X)$ satisfies the following properties:
		\begin{enumerate}
			\item $\Pi$ is monotonically non-decreasing:  $\Pi(f) \ge \Pi(g)$ whenever $f-g \geq 0$.
			\item For every $c \in \mathbb{R}$ and $f \in C(X)$, $\Pi(f+c) = \Pi(f) + c$.
			\item $\Pi$ is $1$-Lipschitz with respect to the $\|\cdot\|_\infty$-norm on $C(X)$. In particular, it is continuous.
			\item $\Pi$ is well defined on $\Gamma$-cohomology classes. Namely, for every $f,f_0, \in C(X)$ and $g \in \Gamma$,
			\[\Pi(f_0 + f -f \circ g)= \Pi(f_0).\]
			In particular, $\Pi$ is $\Gamma$-invariant in the sense that $\Pi(f)=\Pi(f \circ g)$ for every $f \in C(X)$ and $g \in \Gamma$.
			\item $\Pi$ is a convex function.
			\item If $\mu \in \Prob_\Gamma(X)$ is an equilibrium state for $f_0 \in C(X)$ with respect to $\Sigma$, then $\mu$ is a tangent functional of $\Pi$ at $f_0$.
			\item Conversely, a tangent functional of $\Pi$ at any $f \in C(X)$,  is unital, positive and $\Gamma$-invariant.
			If the sofic measure-theoretic entropy is an upper semi-continuous and concave function on $\Prob(X)$, any tangent functional of $\Pi$ at any $f \in C(X)$ is furthermore an equilibrium state for $f_0 \in C(X)$ with respect to $\Sigma$.
		\end{enumerate}
	\end{proposition}
	
	\begin{proof}
		If $h_\Sigma(\Gamma \curvearrowright X)$ is finite, from the variational principle we have \[ \left\lvert P_{\Sigma}(f,\Gamma\curvearrowright X)\right\rvert \leq  h_{\Sigma}(\Gamma\curvearrowright X) + \sup_{\mu \in \Prob_{\Gamma}(X)}\left\lvert \int f\dd\mu \right\rvert \leq h_\Sigma(\Gamma \curvearrowright X) + \norm{f}_{\infty}. \]
		From this inequality it follows that $P_{\Sigma}(f,\Gamma\curvearrowright X)$ is finite for every $f \in C(X)$. Let us remark that in the proof of properties 1 -- 3 below we do not need to assume that $P_{\Sigma}(f,\Gamma\curvearrowright X) \in \mathbb{R}$. In fact, the finiteness of $P_{\Sigma}(f,\Gamma\curvearrowright X)$ also follows from the assumption that $h_\Sigma(\Gamma \curvearrowright X) \in \RR$ together with property $3$.
		\begin{enumerate}
			\item Monotonicity of $\Pi$ follows directly from monotonicity of $f \mapsto P(U,\sigma,\varepsilon,f)$.
			\item This follows directly from the fact that for any $E \subset X^{V_i}$, $f \in C(X)$ and $c \in \mathbb{R}$ we have
			\[\sum_{\varphi \in E} \exp\left(\sum_{ v \in V_i} f(\varphi(v))+c  \right) = e^{|V_i|c} \sum_{\varphi \in E} \exp\left(\sum_{ v \in V_i} f(\varphi(v))+c  \right),\]
			thus for any open set $U$ of probability measures on $X^{\Gamma}$,
			\[P(U,\sigma_i,\varepsilon,f+c)=P(U,\sigma_i,\varepsilon,f)+c.\]
			The result follows by taking $i \to \infty$, $\delta \to 0$ and $\varepsilon \to 0$.
			\item By the two previous properties, for $f_1,f_2 \in C(X)$
			\[  \Pi(f_1)-\Pi(f_2) \le \Pi(f_2 + \|f_1-f_2\|_
			\infty) - \Pi(f_2) = \|f_1 - f_2\|_\infty.\]
			Interchanging the roles of $f_1$ and $f_2$ we conclude that $|\Pi(f_1)-\Pi(f_2)| \le \|f_1 - f_2\|_\infty$.
			\item 
			Suppose $f,f_0, \in C(X)$ and $g\in \Gamma$.
			Since $\int f_0 \dd\mu = \int (f_0+f -f \circ g )\dd\mu$ for any $\mu \in \Prob_\Gamma(X)$, it follows that 
			\[P_\Sigma(\Gamma \curvearrowright X,\mu,f_0)=P_\Sigma(\Gamma \curvearrowright X,\mu,f_0+f -f \circ g )\]
			So by the variational principle 

			\[\Pi(f_0 +f -f\circ g) = \Pi(f_0).\]
			\item Convexity of $\Pi$ follows easily from H\"older's inequality (see~\cite[Proposition 6.1]{chung_2013}).  Alternatively it follows directly from the variational principle, which shows the value of $\Pi$ is a pointwise supremum of affine functions.
			\item Suppose that $\mu \in \Prob_{\Gamma}(X)$ is an equilibrium state for $f$ with respect to $\Sigma$.
			This means that 
			\[ \Pi(f_0) = h_\Sigma( \Gamma \curvearrowright X,\mu) + \int_X f_0 \dd\mu.\]
			On the other hand, by the variational principle we have that for any $f \in C(X)$
			\[ \Pi(f) \ge h_\Sigma( \Gamma \curvearrowright X,\mu) + \int_X f \dd\mu.\]
			It follows that for any $f \in C(X)$
			\[\Pi(f) - \Pi(f_0) \ge \int_X f- f_0 \dd\mu.\]
			This shows that $\mu$ is a tangent functional of $\Pi$ at $f_0$.
			\item Suppose that $\eta \in C(X)^*$ is a tangent functional of $\Pi$ at $f_0$.
			This means that for every $f \in C(X)$
			\[ \Pi(f)-\Pi(f_0) \geq  \eta(f-f_0).  \]
			In particular, set  $f=f_0 \pm 1$ and apply the property 2 to conclude that $\eta(1)=1$.
			Because $\Pi$ is monotone it follows that for any nonnegative $g \in C(X)$
			\[\Pi(f_0-g)-\Pi(f_0) \le 0,\]
			So $\eta(-g) \le 0$ so $\eta(g) \ge 0$. This shows $\eta$ corresponds to a positive measure.
			To see that $\eta$ is $\Gamma$-invariant, use the fact that $\Pi(f_0 + f -f\circ g)=\Pi(f_0)$. 
			
			Now further assume that $H=h_\Sigma(\Gamma \curvearrowright X,\cdot)\colon \Prob(X) \to \mathbb{R}$ is upper semi-continuous and concave.
			By hypothesis we have that $h_{\Sigma}(\Gamma \curvearrowright X)$ is finite and thus $-H$ is proper. It follows by the Fenchel--Moreau theorem that $-H$ is equal on $\Prob(X)$ to the Legendre transform of its Legendre transform, which is the sofic topological pressure function $\Pi$.
			This means that  
			\[h_\Sigma(\Gamma \curvearrowright X,\eta) = \inf_{f \in C(X)} P_\Sigma(\Gamma \curvearrowright X,f)- \eta(f).\]
			In particular, 
			\[h_\Sigma(\Gamma \curvearrowright X,\eta) \ge P_\Sigma(\Gamma \curvearrowright X,f_0)- \eta(f_0).\]
			So 
			\[P_\Sigma(\Gamma \curvearrowright X,\eta,f_0) \ge P_\Sigma(\Gamma \curvearrowright X,f_0).\]
			By the variational principle the above inequality is in fact an equality and thus $\eta$ is an equilibrium measure for $f_0$ with respect to $\Sigma$.
	\end{enumerate}\end{proof}
	
	By the previous proposition, whenever the measure-theoretic sofic entropy is an upper semi-continuous and concave function of the measure,  equilibrium measures are precisely the tangent functionals of the convex continuous function $\Pi$. Let us discuss some sufficient conditions for these hypotheses to hold.
	
	With regards to concavity of  measure-theoretic sofic entropy:
	It is well known that when $\Gamma$ is amenable, the measure-theoretic entropy $H$, when restricted to the invariant measures $\Prob_\Gamma(X)$, is an affine function. Recall that $H(\mu)= -\infty$ for $\mu \not\in \Prob_\Gamma(X)$, so for any action of an amenable group $\Gamma$, we have that $-H= -h_\Sigma(\Gamma \curvearrowright X,\cdot)$ is convex. When $\Gamma$ is non-amenable, it can happen that $\mu_1,\mu_2 \in \Prob_\Gamma(X)$ satisfy $h_\Sigma(\Gamma \curvearrowright X,\mu_i)\ge 0$ for $i=1,2$ but $h_\Sigma(\Gamma \curvearrowright X,\frac{1}{2}\mu_1+\frac{1}{2}\mu_2) = -\infty$, see for instance Proposition 3.1 of~\cite{Bowen_2020}. 

	With regards to upper semi-continuity of  the measure-theoretic sofic entropy, we now recall that \define{expansiveness} is a sufficient condition.
	Recall that an action $\Gamma \curvearrowright X$ is \define{expansive} if there exists $\varepsilon>0$ such that for any $x_1 \ne x_2$ in $X$ there exists $g \in \Gamma$ such that the distance between $gx_1$ and $gx_2$ is at least $\varepsilon$. A positive constant $\varepsilon>0$ as above is called an \define{expansive constant} for $\Gamma \curvearrowright X$. Whether a particular $\varepsilon>0$ is an expansive constant depends on the choice of metric, but the existence of an expansive constant does not.     
	If $\Gamma \curvearrowright X$ is expansive then with respect to any sofic approximation the measure-theoretic sofic entropy is an upper semi-continuous function of the measure, see Theorem $2.1$ of \cite{chungZhang205weakExpansiveness}.
	This is result can be obtained as a consequence of  the fact that when  $\Gamma \curvearrowright X$ is expansive then there exists $\varepsilon >0$ such that
	for all $f \in C(X)$ and $\mu \in \Prob_{\Gamma}(X)$
	\[P_{\Sigma}(\Gamma \curvearrowright X,\mu,f )=P_{\Sigma}(\varepsilon,\Gamma \curvearrowright X,\mu,f ).\]
	Upper semi-continuity of the measure-theoretic entropy for expansive actions now follows using \cref{prop:exp_upper_semi_cont_enropy}.

	\subsection{Sofic pressure for shift spaces}
	
	As discussed above, the sofic topological  pressure with respect to $\Sigma$ does not depend upon the choice of a continuous dynamically generating pseudometric. In particular, for a subshift $X \subset A^{\Gamma}$ we may choose $\rho \colon X \times X \to \RR$ given by \[ \rho(x,y) = \begin{cases}
		0 & \mbox{ if } x(1_{\Gamma}) = y(1_{\Gamma}),\\
		1 & \mbox{ if } x(1_{\Gamma}) \neq y(1_{\Gamma}).
	\end{cases}   \]
	
	It is clear that with this pseudometric, the notion of being $\varepsilon$-separated is exactly the same for every value $\varepsilon<1$. Namely, it is equivalent to the statement that for every $\varphi_1,\varphi_2 \in X^{V_i}$ there is $v \in V_i$ such that $\varphi_1(v)_{1_{\Gamma}} \neq \varphi_2(v)_{1_{\Gamma}}$. Thus for subshifts both definitions simplify to
	
	\[ P_{\Sigma}(f,\Gamma \curvearrowright X) = \inf_{\delta>0} \limsup_{i \to \infty} \frac{1}{|V_i|}\log\left(P(f,N_{\delta}(X^*),\sigma_i,\tfrac{1}{2})\right). \]
	\[ P_{\Sigma}(f,\Gamma \curvearrowright (X,\mu)) = \inf_{\delta>0} \limsup_{i \to \infty} \frac{1}{|V_i|}\log\left(P(f,N_{\delta}(\mu^*),\sigma_i,\tfrac{1}{2})\right). \]
	
	We can further simplify the formulae for pressure by replacing $X^{V_i}$ by the finite set $A^{V_i}$ and writing for $U$ an open set of probability measures on $X$ \[  P(f,U,\sigma) = \frac{1}{|V|}\log \left( \sum_{w \in A^{v_i} : \xi_{\sigma}(w) \in U} \exp\left( \sum_{v \in V} f( \xi_{\sigma,v}(w)  ) \right) \right).   \]
	
	The advantage of the previous definition is that we no longer need to consider a supremum over $\varepsilon$-separated sets nor make computations on the space of orbits $(X^{\star},\mu^{\star})$. It is clear that any $f \in C(X)$ and $\mu \in \Prob_{\Gamma}(X)$ we have
	
	\[ P_{\Sigma}(f,\Gamma \curvearrowright X) = \inf_{\delta>0} \limsup_{i \to \infty} P(f,N_{\delta}(X),\sigma_i). \]
	\[ P_{\Sigma}(f,\Gamma \curvearrowright (X,\mu)) = \inf_{\delta>0} \limsup_{i \to \infty} P(f,N_{\delta}(\mu),\sigma_i). \]
	
	For the second equality, we use the fact that for every $w \in A^{V}$ we have,
	\begin{equation}\label{eq:empric_intergal}
		\sum_{v \in V}f(\xi_{\sigma,v}(w))  = |V| \int_X f\  \dd\left(\xi_{\sigma}(w)\right).
	\end{equation}
	
	For convenience in what follows, if $W$ is a set of probability measures on $A^\Gamma$ define:
	\[ P_{\Sigma}(f,W,\Gamma \curvearrowright X) = \inf_{U \supset W}\limsup_{i \to \infty} P(f,U,\sigma_i),\]
	where the infimum is over all open sets of probability measures $U$ that contain $W$.
	
	\section{Proof of the sofic Lanford--Ruelle theorem for locally constant functions}\label{sec:prooflocal}
	
	For the remainder of this section, we fix a subshift $X \subset A^{\Gamma}$. Recall that we denote by $\cT(X)$ the asymptotic relation of $X$ and by $\cT^0(X)$ the \'{e}tale asymptotic relation of $X$.
	
	For $\delta >0$, a Borel subset $A \subset X$ and $t \in [0,1]$ let
	\[
	N_{\delta}(A,t) \isdef\left\{ \mu \in N_\delta(X) : |\mu(A) - t| \leq \delta\right\}.
	\]
	and for $x,y \in X$, $F \Subset \Gamma$ and $r \geq 0$,
	\[
	N_\delta\left[ (x,y),t,r,F\right] \isdef \left\{ \mu \in N_\delta([x]_F \cup [y]_F,t) : \left| \frac{\mu([y]_{F})}{\mu([x]_{F})} - r \right| \leq \delta
	\right\}.
	\] 
	
	For $r\geq 0$, $C \in \mathbb{R}$ let $r^* = \frac{r}{1+r} \in [0,1)$ and
	\[
	\hat p(r,C) \isdef H(r^*) + r^*C
	\]
	where $H\colon [0,1]\to \RR$ is defined by $H(r^*) = -r^*\log(r^*)-(1-r^*)\log(1-r^*)$ with the usual convention that $0 \log(0) = 0$.
	Note that for fixed $C \in \mathbb{R}$, $\hat{p}(r,C)$ admits a unique global (and local) maximum which is attained at $r=\exp(C)$.
	
	\begin{definition}
		Let $F_1, F_2 \Subset \Gamma$. We say that a subset $X_0 \subset A^{\Gamma}$ has \define{trivial $F_1$-overlaps within $F_2$} if %
		\[(gy)|_{F_2 \cap g F_2}\ne x|_{F_2 \cap gF_2} \mbox{ whenever } g \in F_2F_1^{-1}\setminus \{1_\Gamma\} \mbox{ and }x,y \in X_0.\]
	\end{definition}
	
	Suppose that $F_1 \subset F_2$. In this case the condition above is just a local way to state that for any $z \in A^{\Gamma}$ and $g \in \Gamma\setminus \{1_{\Gamma}\}$ such that $z|_{F_2}=x|_{F_2}$ and $(gz)|_{F_2}=y|_{F_2}$, we must necessarily have that $F_2 \cap gF_1 = \varnothing$. In other words, the patterns $x|_{F_2}$ and $y|_{F_2}$ can only occur in such a way that the ``$F_1$-center'' of $y$ does not intersect the ``$F_2$-center'' of $x$. Next we will show that this condition provides a way to define an endomorphism of $X$ which replaces all occurrences of a pattern $y|_{F_2}$ by another pattern $x|_{F_2}$ whenever they are exchangeable and $x|_{F_2\setminus F_1} = y|_{F_2\setminus F_1}$.
	
	\begin{lemma}\label{lem:pattern_exchange}
		Let $F_1\subset F_2 \Subset \Gamma$, $(x,y)\in \cT^{0}_{F_1}(X)$ and suppose that $\{x,y\}$ has trivial $F_1$-overlaps within $F_2$. There exists a continuous $\Gamma$-equivariant map $\pi \colon A^\Gamma \to A^\Gamma$ such that $\pi(X)\subset X$ and for every $z \in A^{\Gamma}$,
		\begin{enumerate}
			\item If $z|_{F_2}\in \{x|_{F_2}, y|_{F_2}\}$, then $\pi(z)|_{F_2} = x|_{F_2}$.
			\item $\pi(z)|_{F_2} \neq y|_{F_2}$.
		\end{enumerate}
	\end{lemma}
	
	\begin{proof}
		Since $(x,y) \in \cT^0_{F_1}(X)$, it follows that $x|_{F_2 \setminus F_1} = y|_{F_2 \setminus F_1}$. Let us define $\pi\colon A^\Gamma \to A^\Gamma$ as follows. For $z \in A^{\Gamma}$,
		
		\[  \pi(z)(g) = \begin{cases}
			x(h) & \mbox{ if there is } h \in F_1 \mbox{ such that } (hg^{-1}z)|_{F_2} = y|_{F_2}, \\
			z(g) & \mbox{ otherwise. }
		\end{cases}  \]
		
		Let us first show that $\pi$ is well defined. Let $z \in A^{\Gamma}$ and suppose there is $g \in \Gamma$ and $h,h'\in F_1$ such that $(hg^{-1}z)|_{F_2} = y|_{F_2} = (h'g^{-1}z)|_{F_2}$. Let $z' = h'g^{-1}z$ and $\gamma = h'h^{-1}$, then we have $(\gamma^{-1}z')|_{F_2} = z'|_{F_2} = y|_{F_2}$. From $F_1 \subset F_2$ we get that $\gamma\in F_1F_1^{-1}\subset F_2F_2^{-2}$, while from the previous relation we get that $y|_{F_2 \cap \gamma F_2} = (\gamma y)|_{F_2 \cap \gamma F_2}$. As $\{x,y\}$ (and in particular $\{y\}$) has trivial $F_1$-overlaps within $F_2$ we conclude that $\gamma = 1_{\Gamma}$ and thus $h =h'$.
		
		Now let $z \in A^{\Gamma}$. If $z|_{F_2} \in \{x|_{F_2},y|_{F_2}\}$ it follows by definition that $\pi(z)|_{F_1} = x|_{F_1}$. Now let $g \in F_2 \setminus F_1$, if there were $h \in F_1$ such that $(hg^{-1}z)|_{F_2} = y|_{F_2}$, we would have that $(gh^{-1}y)|_{F_2 \cap gh^{-1}F_2} \in \{ x|_{F_2 \cap gh^{-1}F_2},y|_{F_2 \cap gh^{-1}F_2}\}$ which cannot hold as $gh^{-1}\in F_2F_1^{-1}\setminus \{1_{\Gamma}\}$ and $\{x,y\}$ has trivial $F_1$-overlaps within $F_2$. From the argument above we deduce that we must have $\pi(z)(g) = z(g)$ for $g\in F_2\setminus F_1$ and thus $\pi(z)|_{F_2} = x|_{F_2}$.
		
		Now suppose $\pi(z)|_{F_2}=y|_{F_2}$. By the previous argument, we couldn't have had $z|_{F_2} = y|_{F_2}$ and thus there must be $g \in F_2$ such that $\pi(z)(g)\neq z(g)$. This implies that there is $h \in F_1$ such that $(hg^{-1}z)|_{F_2} = y|_{F_2}$ and thus $(hg^{-1}\pi(z))|_{F_2} = x|_{F_2}$. This again cannot occur as $gh^{-1}\in F_2F_1^{-1}$ and $\{x,y\}$ has trivial $F_1$-overlaps within $F_2$.
		
		It is clear from the definition that $\pi$ is continuous and $\Gamma$-equivariant. Now let $z \in X$. Since $(x,y) \in \cT^0_{F_1}(X)$ and $F_1\subset F_2$ it follows that $\xi \colon A^{\Gamma} \to A^{\Gamma}$ given by \[ \xi(z) = \begin{cases} z|_{\Gamma \setminus F_2} \vee x|_{F_2} & \mbox{ if } z|_{F_2} = y|_{F_2}\\
			z & \mbox{ otherwise},
		\end{cases}\]
		is a continuous self-map on $X$. It is clear that the fact that $\{x,y\}$ has trivial $F_1$-overlaps within $F_2$ implies that the collection of maps $\{g^{-1} \xi g\}_{g \in \Gamma}$ pairwise commute. Thus for any finite finite $F \subset \Gamma$ the composition $\xi_F := \prod_{g \in F}g^{-1} \xi g$ is a well defined self-map of $X$ (independent of the order of composition). For every $x \in X$ and any finite $F \subset \Gamma$, $\pi(x)|_F$ coincides with $\xi_{F'}(x)|_F$ for sufficiently big $F'$. From here it follows that $\pi(X)\subset X$.
	\end{proof}

	Let us recall that a function $f\colon X \to \RR$ is said to be $F$-locally constant if $f(x)=f(y)$ for every $x,y \in X$ such that $x|_F = y|_F$. In the following lemma we shall use the usual little $o$ notation $o(g(n))$ to denote a function which goes to zero as $n$ goes to infinity when divided by $g(n)$.

	\begin{lemma}\label{lem:opt_ratio}
		Let $F_1\subset F_2 \Subset \Gamma$ such that $F_1F_1^{-1}\subset F_2$, $(x,y) \in \cT^0_{F_1}(X)$, $t \in [0,1]$ and $f\colon X \to \mathbb{R}$ an $F_1$-locally constant function.
		Suppose that $\{x,y\}$ has trivial $F_1$-overlaps within $F_2$ and that for all $\delta >0$ there are infinitely many $i \in \NN$ such that
		\[
		\xi_{\sigma_i}^{-1}(N_\delta([x]_{F_2} \cup [y]_{F_2},t)) \ne \varnothing.
		\]
		
		Then for all $r_1,r_2\geq 0$ we have
		\begin{align}
			\inf_{\delta >0}
			\limsup_{i \to \infty}
			\left(P(f,N_\delta\left[(x,y),t,r_1,F_2\right],\sigma_i)- P(f,N_\delta\left[(x,y),t,r_2,F_2\right],\sigma_i) \right)\\ 
			= t\left( \hat p (r_1,\Psi_f(x,y)) - \hat p (r_2,\Psi_f(x,y))   \right). 
		\end{align}
	\end{lemma}
	
	\begin{proof}
		Let $F_1,F_2 \Subset \Gamma$, $(x,y) \in \cT^0_{F_1}(X)$, $t \in [0,1]$ and $f\colon X \to \mathbb{R}$ as in the statement.~\Cref{lem:pattern_exchange} provides the existence of a continuous $\Gamma$-equivariant map $\pi\colon A^{\Gamma} \to A^{\Gamma}$ such that $\pi(X)\subset X$ and which replaces any occurrence of the pattern $y|_{F_2}$ by the pattern $x|_{F_2}$ and which contains no occurrences of the pattern $y|_{F_2}$.

			For every $i \in \NN$ let $\pi_i\colon A^{V_i} \to A^{V_i}$ be the map given by \[\pi_i(w)(v)=(\pi(\xi_{\sigma_i,v}(w)))({1_\Gamma}).\] 
			
			For $i \in \mathbb{N}$, $\delta >0$ and $r \geq 0$ let
			
			\begin{align}
				\Omega_i(\delta) & = \xi_{\sigma_i}^{-1} (N_{\delta}([x]_{F_2} \cup [y]_{F_2},t)),\\
				\Omega_{\pi,i}(\delta) & = \pi_i(\xi_{\sigma_i}^{-1}(N_{\delta}([x]_{F_2}\cup [y]_{F_2},t))),\\
				\Omega_i(\delta,r) & = \xi_{\sigma_i}^{-1} (N_{\delta}[(x,y),t,r,F_2]).
			\end{align}

			For $w \in A^{V_i}$, let us denote by $S_i(f,w)= \sum_{v \in V_i}f(\xi_{\sigma_i, v}(w))$. Our main goal is to estimate the expression \[
			P(f,N_\delta\left[(x,y),t,r,F_2\right],\sigma_i)=
			\frac{1}{|V_i|} \log \left(\sum_{w \in  \Omega_i(\delta,r)}  \exp({S_i(f,w)}) \right). \]
			
			In order to do that, let us first fix $w \in \Omega_i(\delta,r)$. We shall first show that the difference $S_i(f,w)- S_i(f,\pi_i(w))$ is close to $r^*t|V_i|\Psi_f(x,y)$. Intuitively, this happens because if $v \in V_i$ is a sufficiently good approximation (in the sense that the configurations $\xi_{\sigma_i, v}(w)$ and $\xi_{\sigma_i, v}(\pi(w))$ ``look like elements of $X$'' restricted to some large set), then the term $f(\xi_{\sigma_i, v}(w))-f(\xi_{\sigma_i, v}(\pi(w)))$ is non-zero only if $\xi_{\sigma_i, v}(w)|_{F_2} = y|_{F_2}$, which occurs with probability close to $r^*t$ due to $w$ being in $\Omega_i(\delta,r)$.
			
			Let us now proceed formally. As $\Sigma$ is a sofic approximation sequence for $\Gamma$, the proportion of vertices $v\in V_i$ which are not $F_1F_2$-good is $o(|V_i|)$ and thus \[ \sum_{v \in V_i \mbox{ is not } F_1F_2 \mbox{ good }}\left(f(\xi_{\sigma_i,v}(w)) - f(\xi_{\sigma_i,v}(\pi_{i}(w))) \right) \leq 2o(|V_i|)\norm{f}_{\infty} = o(|V_i|). \]
			
			Suppose that $v \in V_i$ is $F_1F_2$-good for $\sigma_i$. Using the fact that $f$ is $F_1$-local and the definition of $\pi_i$, we obtain that, as long as there is no $g \in F_1$ such that $\xi_{\sigma_i,\sigma_i(g)v}(w)|_{F_2}=y|_{F_2}$, then
			\[f(\xi_{\sigma_i,v}(w)) - f(\xi_{\sigma_i,v}(\pi_{i}(w))) = 0.\]
			
			Therefore we only need to consider the $v$ for which there is some is $g \in F_1$ such that $\xi_{\sigma_i,\sigma_i(g)v}(w)|_{F_2}=y|_{F_2}$. Note that as $v$ is $F_1F_2$-good and $\{x,y\}$ have trivial $F_1$-overlaps within $F_2$, it follows that such a $g$ is unique and we have \[ \sum_{h \in F_1F_1^{-1}} (f(\xi_{\sigma_i,\sigma_i(gh)v}(w)) -f(\xi_{\sigma_i,\sigma_i(gh)v}(\pi(w))) = \Psi_f(x,y).  \]     
			
			Given $\delta>0$ as above, let $r^* = \frac{r}{1+r}$ and define $\delta^{-}$ and $\delta^{+}$ by
			\[  \delta^{-} = r^*-\frac{r - \delta}{1 + r - \delta}, \ \delta^{+}= \frac{r + \delta}{1 + r + \delta} - r^*.     \]
			
			As $w \in \Omega_i(\delta,r)$, it follows that the number $N_i$ of $v \in V_i$ such that $\xi_{\sigma_i,\sigma_i(gh)v}(w)|_{F_2} =y|_{F_2}$ is bounded above and below as follows \[   (r^*-\delta^{-})(t-\delta)(|V_i|-o(V_i)) \leq N_i \leq (r^*+\delta^{+})(t+\delta)|V_i|. \]
			
			Putting all of the previous estimates together, we obtain that
			
			\[  o(|V_i|) + (r^* - \delta^{-})(t- \delta)|V_i|\Psi_f(x,y)\leq S_i(f,w)-S_i(f,\pi_i(w)) \leq  o(|V_i|) + (r^* + \delta^{+})(t+ \delta)|V_i|\Psi_f(x,y).  \]

			For $w' \in \Omega_{\pi,i}(\delta)$, let us define $K_{w'}(\delta,r) = \left|\pi_i^{-1}\left( \{w'\}\right) \cap \Omega_i(\delta,r) \right|$ to be the number of $w \in \Omega_i(\delta,r)$ such that $\pi_i(w) = w'$. Note that we may always write
			
			\begin{align}
				\sum_{w \in  \Omega_i(\delta,r)} \exp(S_i(f,w))  & = \sum_{w' \in  \Omega_{\pi,i}(\delta)} \left(\sum_{w \in \Omega_i(\delta,r)\ :\ \pi_i(w)=w' } \exp(S_i(f,w')+ S_i(f,w)-S_i(f,\pi_i(w))) \right).
			\end{align}
			
			Using the bounds we obtained for $S_i(f,w)-S_i(f,\pi_i(w))$, we get that \begin{align}   \sum_{w \in  \Omega_i(\delta,r)} \exp(S_i(f,w)) & \leq \sum_{w' \in  \Omega_{\pi,i}(\delta)} K_{w'}(\delta,r)\exp(S_i(f,w') + o(|V_i|) + (r^* + \delta^{+})(t+ \delta)|V_i|\Psi_f(x,y))\\
				&  \leq \sum_{w' \in  \Omega_{\pi,i}(\delta)} \exp(S_i(f,w') +\log(K_{w'}(\delta,r))+ o(|V_i|) + (r^* + \delta^{+})(t+ \delta)|V_i|\Psi_f(x,y)), \end{align}
			and
			\begin{align}  \sum_{w \in  \Omega_i(\delta,r)} \exp(S_i(f,w)) & \geq \sum_{w' \in  \Omega_{\pi,i}(\delta)} K_{w'}(\delta,r)\exp(S_i(f,w') + o(|V_i|) + (r^* - \delta^{-})(t-\delta)|V_i|\Psi_f(x,y))\\
				& \geq \sum_{w' \in  \Omega_{\pi,i}(\delta)}\exp(S_i(f,w') + \log(K_{w'}(\delta,r)) + o(|V_i|) + (r^* - \delta^{-})(t-\delta)|V_i|\Psi_f(x,y)).
			\end{align}

			Our next goal is to provide an estimate for $\log(K_{w'}(\delta,r))$. Let $G(w',x) \isdef \left\{ v \in V_i : \xi_{\sigma_i, v}(w')|_{F_2}  = x|_{F_2} \right\}$. By our assumptions we have that for infinitely many values of $i$, \[ (t - \delta)|V_i|+ o(|V_i|) \leq |G(w',x)| \leq (t + \delta)|V_i|+ o(|V_i|).  \]
			
			Notice that any $w\in A^{V_i}$ such that $\pi_i(w)=w'$ is uniquely determined by the subset of vertices in $G(w',x)$, which consists of the positions that initially held the pattern $y|_{F_2}$ and were erased by the map $\pi_i$. Using again that $w \in \Omega_i(\delta,r)$ we obtain that
			
			\[  \binom{G(w',x)}{\lfloor r^*G(w',x)\rfloor}  \leq K_{w'}(\delta,r) \leq \sum_{k = \lceil (r^*-\delta^{-})G(w',x)\rceil}^{\lfloor (r^*+\delta^{+})G(w',x)\rfloor}\binom{G(w',x)}{k}.\]
			
			It is easy to show using Stirling's approximation that whenever $n \in \NN$ and $\alpha \in (0,\frac{1}{2})$, then \[ \log\binom{ n}{\lfloor\alpha n \rfloor } = (1+o(1))H(\alpha)n. \]
			
			This immediately yields the lower bound \[ K_{w'}(\delta,r) \geq \binom{G(w',x)}{\lfloor r^*G(w',x)\rfloor} =\exp( (1+o(1))H(r^*)(t-\delta)|V_i| +o(|V_i|)). \]
			
			For the upper bound, let $s \in (-\delta^{-},\delta^{+})$ such that the binomial coefficient \[\binom{G(w',x)}{\lfloor(r^*+s)G(w',x)\rfloor}\] 
			is largest. We obtain the following upper bound for small enough $\delta$,
			
			\begin{align}
				K_{w'}(\delta,r) & \leq \lceil(\delta^{+}-\delta^{-})|G(w',x)|\rceil \exp( (1+o(1))H(r^*+s)(t+\delta)|V_i| +o(|V_i|))\\
				& \leq |V_i| \exp( (1+o(1))H(r^*+s)(t+\delta)|V_i| +o(|V_i|)).
			\end{align}
			
			Putting these two bounds together and taking logarithms, we obtain 
			
			\[ o(|V_i|) + (t - \delta)H(r^*)|V_i| \leq \log(K_{w'}(\delta,r)) \leq  o(|V_i|) + (t + \delta)H(r^* +s)|V_i|. \]
			
			Now we can refine our previous bounds. Putting this last computation back in the previous formulas we get
			
			\begin{align}
				\frac{1}{|V_i|} \log \left(\sum_{w \in  \Omega_i(\delta,r)}  \exp(S_i(f,w)) \right) & \leq  \frac{1}{|V_i|} \log \left(\sum_{w' \in  \Omega_{\pi_i}(\delta)}\exp(S_i(f,w'))\right)   \\ 
				&  + \frac{o(|V_i|)}{|V_i|}+(t + \delta)\left(H(r^* +s) + (r^* + \delta^{+})\Psi_f(x,y)\right).
			\end{align}
			
			and
			
			\begin{align}
				\frac{1}{|V_i|} \log \left(\sum_{w \in  \Omega_i(\delta,r)}  \exp(S_i(f,w)) \right) & \geq  \frac{1}{|V_i|} \log \left(\sum_{w' \in  \Omega_{\pi_i}(\delta)}\exp(S_i(f,w'))\right)   \\ 
				&  + \frac{o(|V_i|)}{|V_i|}+(t - \delta)\left(H(r^*) + (r^* - \delta^{-})\Psi_f(x,y)\right).
			\end{align}
			
			Let now $r_1,r_2\geq 0$ and notice that the first term in both bounds does not depend on $r$ and thus disappears when plugging the above bounds into the difference
			\begin{align}
				P(f,N_\delta[(x,y),t,r_1,F_2],\sigma_i)- P(f,N_\delta[(x,y),t,r_2,F_2],\sigma_i).
			\end{align}
			Furthermore, the term $o(|V_i|)/|V_i|$ vanishes when taking the limsup as $i \to \infty$, also, the terms $\delta^{-},\delta^{+}$ and $s$ all go to $0$ uniformly as $\delta$ goes to $0$. We obtain thus that for all $r_1,r_2\geq 0$ we have
			\begin{align}
				\inf_{\delta >0}
				\limsup_{i \to \infty}
				\left(P(f,N_\delta\left[(x,y),t,r_1,F_2\right],\sigma_i)- P(f,N_\delta\left[(x,y),t,r_2,F_2\right],\sigma_i) \right)\\ 
				= t\left( H(r_1^*)+r_1^*\Psi_f(x,y) - ( H(r_2^*)+r_2^*\Psi_f(x,y))   \right).\\
				= t\left( \hat p (r_1,\Psi_f(x,y)) - \hat p (r_2,\Psi_f(x,y))   \right).
			\end{align}
			
			Which is what we wanted to show.\end{proof}
		
		\begin{definition}
			Given $x \in A^\Gamma$ and $F_1 \subseteq F_2 \Subset \Gamma$,
			we say that $x$ is $(F_1,F_2)$-self-overlapping  if 
			\[(gx)|_{(F_2 \setminus F_1) \cap g (F_2 \setminus F_1)}= x|_{(F_2 \setminus F_1) \cap g (F_2 \setminus F_1)} \mbox{ for some } g \in F_2F_1^{-1}\setminus \{1_\Gamma\}.\]
			We say that $x$ is \emph{non self-overlapping} if for every $F_1 \Subset \Gamma$ there exists $F_2 \Subset \Gamma$ such that $F_1 \subseteq F_2$ and $x$ is not $(F_1,F_2)$-self-overlapping.
		\end{definition}
		
		Notice that the condition of being $(F_1,F_2)$-self-overlapping is satisfied automatically if there is $g \in F_2F_1^{-1}\setminus \{1_\Gamma\}$ for which $(F_2 \setminus F_1) \cap g (F_2 \setminus F_1) = \varnothing$. Thus in the definition of non self-overlapping it is understood that the sets $F_2$ must satisfy that $(F_2 \setminus F_1) \cap g (F_2 \setminus F_1) \neq \varnothing$ for every $g \in F_2F_1^{-1}\setminus \{1_\Gamma\}$. The usefulness of this condition is justified in the following lemma.
		
		\begin{lemma}\label{lem:noself_overlap_implies_non_overlap_within}
			Let $x \in A^{\Gamma}$ be non self-overlapping. For every $F_1\Subset \Gamma$ and $y \in A^{\Gamma}$ such that $(x,y) \in \cT_{F_1}(X)$, there exists $F_2\Subset \Gamma$ with $F_1 \subset F_2$ such that $\{x,y\}$ has trivial $F_1$-overlaps within $F_2$.
		\end{lemma}
		
		\begin{proof}
			As $x$ is non self-overlapping, there is $F_2\Subset \Gamma$ such that $F_1 \subseteq F_2$ and $x$ is not $(F_1,F_2)$-self-overlapping. Suppose that $\{x,y\}$ does not have trivial $F_1$-overlaps within $F_2$, then there is $g \in F_2F_1^{-1}\setminus \{1_\Gamma\}$ and $x',y' \in \{x,y\}$ such that \[(gy')|_{F_2 \cap g F_2} = x'|_{F_2 \cap gF_2}.\]
			In particular, as $(F_2 \setminus F_1) \cap g (F_2 \setminus F_1) \neq \varnothing$, we have \[(gy')|_{(F_2 \setminus F_1) \cap g (F_2 \setminus F_1)} = x'|_{(F_2 \setminus F_1) \cap g (F_2 \setminus F_1)}. \]
			Finally, as $(x,y) \in \cT_{F_1}(X)$, it follows that $x|_{F_2 \setminus F_1} = y|_{F_2 \setminus F_1}$ and $(gx)|_{g(F_2\setminus F_1)}=(gy)|_{g(F_2\setminus F_1)}$. We obtain that \[  (gx)|_{(F_2 \setminus F_1) \cap g (F_2 \setminus F_1)} = x|_{(F_2 \setminus F_1) \cap g (F_2 \setminus F_1)}.\]
			Which contradicts the fact that $x$ is not $(F_1,F_2)$-self-overlapping.\end{proof}

		\begin{lemma}\label{lem:noverlap}
			Let $\Sigma$ be a sofic approximation for $\Gamma$ such that  $h_{\Sigma}(\Gamma \curvearrowright X)\geq 0$ and let $\mu$ be an equilibrium measure for a subshift $X\subset A^{\Gamma}$ with with respect to an $F$-locally constant function $f\colon X \to \RR$ and $\Sigma$. Suppose also that $(x,y) \in \cT^0(X)$, $x$ is in the support of $\mu$ and $x$ is non self-overlapping.

			Then $y$ is also  in the support of $\mu$ and there exists a sequence of finite subsets $F_n \Subset \Gamma$ increasing to $\Gamma$ so that \[\lim_{n \to \infty}  \frac{\mu([y]_{F_n})}{\mu([x]_{F_n})} = \exp(\Psi_f(x,y)).\]
		\end{lemma}
		\begin{proof}

			Choose a sequence of finite subsets  $F_1^{(n)} \Subset \Gamma$ which is increasing to $\Gamma$ and such that $F\subset F_1^{(n)}$ and $(x,y) \in \cT^0_{F_1^{(n)}}(X)$ for all $n \in \mathbb{N}$. By~\Cref{lem:noself_overlap_implies_non_overlap_within}, for each $n \in \NN$ there exists  $F_2^{(n)} \Subset \Gamma$ such that $F_1^{(n)} \subset F_2^{(n)}$ and  $\{x,y\}$ has trivial $F_1^{(n)}$-overlaps within $F_2^{(n)}$. Also, because $x$ is in the support of $\mu$, we can define \[t = \mu([x|_{F_2^{(n)}}]\cup [y|_{F_2^{(n)}}])>0 \mbox{ and } r = \frac{\mu([y|_{F_2^{(n)}}])}{\mu([x|_{F_2^{(n)}}])} \geq 0. \]
			For now fix $n \in \NN$ and denote $F_2 := F_2^{(n)}$.
			
			Note that $\mu \in N_{\delta}[(x,y),t,r,F_2] \subset N_{\delta}([x|_{F_2}]\cup [y|_{F_2}],t)$ for every $\delta>0$. From the assumption that $h_{\Sigma}(\Gamma \curvearrowright X)\geq 0$ and the variational principle for sofic pressure (\Cref{thm:variational}) we obtain that $h_{\Sigma}(\Gamma \curvearrowright (X,\mu))\geq 0$. As $\mu \in N_{\delta}[(x,y),t,r,F_2]$, it follows that there is an infinite sequence of $i \in \NN$ such that $\xi_{\sigma_i}^{-1}(N_{\delta}[(x,y),t,r,F_2])\neq \varnothing$.
			
			From~\Cref{lem:opt_ratio} it follows that for every $r_1,r_2 \geq 0$ we have
			
			\begin{align}
				\lim_{\delta \to 0}
				\limsup_{i \to \infty}
				\left(P(f,N_\delta[(x,y),t,r_1,F_2],\sigma_i)- P(f,N_\delta[(x,y),t,r_2,F_2],\sigma_i) \right)\\ 
				= t\left( \hat p (r_1,\Psi_{f}(x,y)) - \hat p (r_2,\Psi_{f}(x,y))   \right).
			\end{align}
			
			As $\mu$ is an equilibrium measure for $f$, for every $\varepsilon>0$ we can find $\delta_2>0$ such that for all sufficiently big $i \in \NN$ we have \[\left|P_{\Sigma}(f,\Gamma \curvearrowright (X,\mu))- P(f, N_{\delta_2}[(x,y),t,r,F_2],\sigma_i)\right| < \varepsilon.\]
			
			Fix any $\varepsilon >0$. It follows that for any $r' \geq 0$ and sufficiently big $i \in \NN$ we have \[ P(f, N_{\delta_2}[(x,y),t,r',F_2],\sigma_i) - P(f, N_{\delta_2}[(x,y),t,r,F_2],\sigma_i) \leq  \varepsilon.  \]
			
			From the above relation we obtain that 
			
			\[ \hat{p}(r',\Psi_{f}(x,y)) - \hat{p}(r,\Psi_{f}(x,y)) \leq \frac{\varepsilon}{t}.   \]
			
			Recall that for any $C \in \RR$, the expression $\hat p(r',C)$ as a function of $r' \geq 0$ has a unique local maximum at $r' = \exp(C)$. We obtain that \[\hat{p}(\exp(\Psi_{f}(x,y)), \Psi_{f}(x,y)) - \hat{p}(r,\Psi_{f}(x,y)) \leq \frac{\varepsilon}{t}.\]
			
			The restriction to the first coordinate of the function $\hat{p}$ which maps $r' \mapsto \hat{p}(r',\Psi_{f}(x,y))$ attains a unique maximum at $r' = \exp((\Psi_{f}(x,y)))$. Moreover, its inverse has two continuous branches which meet on the maximum. It follows that for every $\eta>0$ we can choose $\varepsilon>0$ such that whenever the above equation holds, then $|r-\exp(\Psi_{f_{0}}(x,y))| \leq \eta$. As $\varepsilon$ was arbitrary, we obtain that \[  \frac{\mu([y]_{F_2^{(n)}})}{\mu([x]_{F_2^{(n)}})}  = r = \exp((\Psi_{f}(x,y))). \]
			As $n \in \NN$ was arbitrary, we obtain that
			\[
			\lim_{n \to \infty}  \frac{\mu([y]_{F_2^{(n)}})}{\mu([x]_{F_2^{(n)}})} = \exp(\Psi_f(x,y)),
			\]

			as we required. This equality automatically implies that both $x$ and $y$ are in the support of $\mu$.
			
		\end{proof}
		
		The remainder of the proof consists essentially on getting rid of the non self-overlapping hypothesis. In order to do this, we will first show that the condition occurs almost surely with respect to the uniform Bernoulli measure on the full $\Gamma$-shift, and then we use a trick involving a direct product to take advantage of the above property on an arbitrary subshift.

		\begin{lemma}\label{lem:nooverlap_iid}
			Let $\mu$ be the uniform Bernoulli measure on $\{0,1\}^\Gamma$. There exists a Borel subset $X_0 \subset \{0,1\}^\Gamma$ with $\mu(X_0)=1$ such that every $x \in X_0$ is non self-overlapping.
		\end{lemma}
		\begin{proof}

			For any  $ F_1 \subseteq F_2 \Subset \Gamma$ and $g \in  F_2F_1^{-1}\setminus \{1_\Gamma\}$
			\[
			\mu \left(\left\{x \in \{0,1\}^\Gamma : ~ (gx)|_{(F_2 \setminus F_1) \cap g (F_2 \setminus F_1)}= x|_{(F_2 \setminus F_1) \cap g (F_2 \setminus F_1)}\right\}  \right) = \left( \frac{1}{2} \right)^{|(F_2 \setminus F_1) \cap g (F_2 \setminus F_1)|}.
			\]
			By a simple union bound, it follows that
			\begin{align}
				\mu \left( \left\{ x \in \{0,1\}^\Gamma: x \mbox { is }(F_1,F_2)\mbox{-self-overlapping} \right\} \right)& \le \sum_{ g \in F_2F_1^{-1}\setminus \{1_\Gamma\}} \left( \frac{1}{2} \right)^{|(F_2 \setminus F_1) \cap g (F_2 \setminus F_1)|}\\
				& \leq  |F_1||F_2|\max_{g \in F_2F_1^{-1}}2^{-|(F_2 \setminus F_1) \cap g (F_2 \setminus F_1)|}\\
				& \leq  |F_1|4^{|F_1|}|F_2|\max_{g \in F_2F_1^{-1}}2^{-|F_2 \cap gF_2|}.
			\end{align}
			
			As $\Gamma$ is countably infinite, in order to prove the lemma it suffices to show that for any fixed $F_1 \Subset \Gamma$ there exists a sequence $(F_2^{(n)})_{n \geq 1}$, with $F_1 \subseteq F_2^{(n)} \Subset \Gamma$ for all $n \geq 1$, and with the property that $\mu$-almost surely the number of $n$'s for which $x$ is $(F_1,F_2^{(n)})$-self-overlapping is at most finite.
			By the Borel–Cantelli lemma, it suffices in turn to find for each $F_1 \Subset \Gamma$ a sequence $(F_2^{(n)})_{n \geq 1}$ with $F_1 \subseteq  F_2^{(n)} \Subset \Gamma$ for all $n \geq 1$ such that
			\[
			\sum_{n =1}^\infty \mu\left(\left\{ x \in \{0,1\}^\Gamma: x \mbox { is }(F_1,F_2^{(n)})\mbox{-self-overlapping} \right\} \right) < \infty.
			\]
			By the previous calculation, it suffices to have that 
			\[
			\sum_{n =1}^\infty |F_2^{(n)}|\max_{g \in F_2^{(n)}F_1^{-1}}2^{-|F_2^{(n)} \cap g F_2^{(n)}|} < \infty.
			\]
			Let $F_1 \Subset \Gamma$ be an arbitrary non-empty finite set. If $\Gamma$ is locally finite, then we can choose a sequence $(F_2^{(n)})_{n \geq 1}$ of finite subgroups of $\Gamma$, each containing $F_1$ with $|F_n| > n$. In that case $F_2^{(n)}F_1^{-1}= F_2^{(n)}$ and for any $g \in F_2^{(n)}$ we have $F_2^{(n)}\cap g F_2^{(n)} = F_2^{(n)}$. We deduce that in this case
			\begin{align}
				\sum_{n =1}^\infty |F_2^{(n)}|\max_{g \in F_2^{(n)}F_1^{-1}}2^{-|F_2^{(n)} \cap g F_2^{(n)}|} \leq  \sum_{n =1}^\infty |F_2^{(n)}| 2^{-|F_2^{(n)}|} \leq \sum_{n =1}^\infty n2^{-n} < \infty.
			\end{align}
			
			If $\Gamma$ is not locally finite, there exists $S \Subset \Gamma$ which is a finite symmetric set with $\{1_\Gamma\} \cup F_1 \subset S$ and so that the group generated by $S$ is infinite.
			For any $m \geq 1$ and $g \in S^{3m} F_1^{-1}$ we can choose $h \in \Gamma$ such that $hS^m \subset S^{3m} \cap gS^{3m}$ and thus we have $|S^{3m} \cap gS^{3m}| \geq |S^m|$. As $|S^{3^n}| \le |S|^{3^n}$ for all $n$ it follows that there exists a strictly increasing sequence of integers $(k_n)_{n=1}^\infty$ such that 
			$|S^{3^{k_n-1}}| \ge |S^{3^{k_n}}|^{\frac{1}{ 4}}$ for all $n \in \mathbb{N}$. 
			For a sequence as above, let $F_2^{(n)}= S^{3^{k_n}}$.
			It follows that for each $n \in \mathbb{N}$ and every $g \in F_2^{(n)} F_1^{-1}$,
			\[
			|F_2^{(n)} \cap g F_2^{(n)}| \geq |S^{3^{k_n-1}}|  \ge |S^{3^{k_n}}|^{\frac{1}{ 4}} = |F_2^{(n)}|^{\frac{1}{4}},
			\]
			and in particular
			
			\begin{align}
				\sum_{n =1}^\infty |F_2^{(n)}|\max_{g \in F_2^{(n)}F_1^{-1}}2^{-|F_2^{(n)} \cap g F_2^{(n)}|} & \leq  \sum_{n =1}^\infty |F_2^{(n)}| 2^{-|F_2^{(n)}|^{\frac{1}{4}}} \leq \sum_{n =1}^\infty n 2^{\sqrt[4]{n}} < \infty.
			\end{align}
		\end{proof}

		\begin{lemma}\label{lem:cheating}
			Let $\Sigma$ be a sofic approximation sequence for $\Gamma$. Let $f\in C(X)$ and $\mu$ an equilibrium measure on $X$ for $f$. Consider $\widetilde{X} = X \times \{\symb{0},\symb{1}\}^{\Gamma}$, $\widetilde{f}\colon \widetilde{X} \to \RR$ given by $\widetilde{f}(x,y)=f(x)$ for every $x \in X$ and $y \in \{\symb{0},\symb{1}\}^{\Gamma}$ and let $\widetilde{\mu}$ be the product of $\mu$ and the uniform Bernoulli measure on $\{\symb{0},\symb{1}\}^{\Gamma}$. We have that $\widetilde{\mu}$ is an equilibrium measure for $\widetilde{f}$ on $\widetilde{X}$.
		\end{lemma}
		
		\begin{proof}
			Let $\nu$ denote the the uniform Bernoulli measure on $\{\symb{0},\symb{1}\}^{\Gamma}$. As $(\{\symb{0},\symb{1}\}^{\Gamma},\nu)$ is a Bernoulli shift, it follows by Theorem 6.3 of~\cite{Bowen_2020} that the measure-theoretic sofic entropy of the product system can be expressed as follows, \[ h_{\Sigma}( \Gamma \curvearrowright (\widetilde{X},\widetilde{\mu})) = h_{\Sigma}( \Gamma \curvearrowright ({X},{\mu})) + h_{\Sigma}( \Gamma \curvearrowright (\{\symb{0},\symb{1}\}^{\Gamma},\nu)) =h_{\Sigma}( \Gamma \curvearrowright ({X},{\mu})) + \log(2).   \]
			
			Let $\mu'$ be an arbitrary $\Gamma$-invariant measure on $X$. We argue that for every $\Gamma$-invariant measure $\widetilde{\mu}'$ on $\widetilde{X}$ with marginal $\mu'$ on $X$ we have \[  h_{\Sigma}(\Gamma \curvearrowright (\widetilde{X},\widetilde{\mu}')) \leq h_{\Sigma}( \Gamma \curvearrowright ({X},{\mu'}))+\log(2).   \]
			
			This follows form the fact that the number of $w \in (A \times \{\symb{0},\symb{1}\})^{V_i}$ for which $\xi_{\sigma_i}(w) \in N_{\delta}(\widetilde{\mu}')$ is bounded by $2^{|V_i|}$ times the number of microstates $u \in A^{V_i}$ such that $\xi_{\sigma_i}(u) \in N_{\delta}({\mu}')$. 
			
			Using the definition of $\widetilde{f}$, we get that for every $\Gamma$-invariant measure $\widetilde{\mu}'$ with marginal $\mu'$ on $X$, we have \[ \int \widetilde{f}\dd \widetilde{\mu}' = \int{f}\dd {\mu'}.  \] 
			
			Putting together all the previous computations, we get
			\begin{align}
				P_{\Sigma}(\widetilde{f}, \Gamma \curvearrowright (\widetilde{X},\widetilde{\mu}')) & = h_{\Sigma}(\Gamma \curvearrowright (\widetilde{X},\widetilde{\mu}')) + \int \widetilde{f}\dd \widetilde{\mu}'\\
				& = h_{\Sigma}(\Gamma \curvearrowright (\widetilde{X},\widetilde{\mu}')) + \int f\dd {\mu}'\\
				& \leq \log(2) +  h_{\Sigma}(\Gamma \curvearrowright ({X},{\mu}')) + \int f\dd {\mu}' \\
				& \leq \log(2) + h_{\Sigma}(\Gamma \curvearrowright ({X},{\mu})) + \int f\dd {\mu} \\
				& = h_{\Sigma}( \Gamma \curvearrowright (\widetilde{X},\widetilde{\mu})) + \int \widetilde{f}\dd \widetilde{\mu}\\
				& = P_{\Sigma}(\widetilde{f}, \Gamma \curvearrowright (\widetilde{X},\widetilde{\mu}))
			\end{align}
			
			And therefore $\widetilde{\mu}$ is an equilibrium measure for $\widetilde{f}$ on $\widetilde{X}$.\end{proof}
		
		The following result is the main theorem in the case where the function $f$ is locally constant. We shall first show it in this case and then generalize it to larger classes of functions in~\Cref{sec:convexgeneralization}.
		
		\begin{theorem}\label{thm:LR_etale_local}
			Let $\Sigma$ be a sofic approximation sequence for $\Gamma$. Consider a subshift $X\subset A^{\Gamma}$ with $h_{\Sigma}(\Gamma \curvearrowright X)\geq 0$. Let $F\Subset \Gamma$ and suppose $f\colon X \to \RR$ is an $F$-locally constant function and $\mu$ an equilibrium measure on $X$ for $f$. Then $\mu$ is \'etale Gibbs with respect to $f$.
		\end{theorem}
		
		\begin{proof}
			Let $(x,y)\in \cT^0(X)$ such that $x$ is in the support of $\mu$. Let $\widetilde{X},\widetilde{f}$ and $\widetilde{\mu}$ as in~\Cref{lem:cheating}. Then $\widetilde{\mu}$ is an equilibrium measure for $\widetilde{f}$ on $X$. Now let $(x',y')\in \cT(\{\symb{0},\symb{1}\}^{\Gamma}) = \cT^0(\{\symb{0},\symb{1}\}^{\Gamma})$ such that $x'$ is non self-overlapping (its existence assured by~\Cref{lem:nooverlap_iid}).

			Then it follows that $((x,x'),(y,y')) \in \cT^0(\widetilde{X})$, $(x,x')$ is in the support of $\widetilde{\mu}$, and $(x,x')$ is non self-overlapping. %
			By~\Cref{lem:noverlap} there exists a sequence of finite sets $F_n \Subset \Gamma$ increasing to $\Gamma$ such that 
			\[ \lim_{n \to \infty} \frac{ \widetilde{\mu}([(y,y')]_{F_n})   }{ \widetilde{\mu}([(x,x')]_{F_n})   }= \exp(\Psi_{\widetilde{f}} ((x,x'),(y,y'))). \]

			Let $\nu$ be the uniform Bernoulli measure on $\{\symb{0},\symb{1}\}^{\Gamma}$. By definition we have $\widetilde{\mu}=\mu \times \nu$, therefore on the one hand we get for any $K \Subset \Gamma$ \[  \frac{ \widetilde{\mu}([(y,y')]_K)   }{ \widetilde{\mu}([(x,x')]_K)   } = \frac{ {\mu}([y]_K)\nu([y']_K)   }{ {\mu}([x]_K) \nu([x']_K)  } = \frac{ {\mu}([y]_K)   }{ {\mu}([x]_K)  }. \]
			On the other hand, using that $\widetilde{f}(x,x')=f(x)$ for every $(x,x')\in \widetilde{X}$, we get \[ \exp(\Psi_{\widetilde{f}} ((x,x'),(y,y'))) = \exp(\Psi_{f} (x,y)).  \]
			Hence we conclude that \[ \lim_{n \to \infty} \frac{ {\mu}([y]_{F_n})   }{ {\mu}([x]_{F_n})  } =  \exp(\Psi_{f} (x,y)).  \]
			For every $x$ in the support of $\mu$. By~\Cref{lemma:RN-cocycle-form-MCT} it follows that there is a Borel set $X'\subset X$ with $\mu(X\setminus X')=0$ such that for every $(x,y)\in X \times X' \cap \cT^0(X)$ we have
			
			\[ D_{\mu,\cT^0(X)}(x,y) = \lim_{F \nearrow \Gamma}\frac{ {\mu}([y]_{F})   }{ {\mu}([x]_{F})}. \]
			
			From the above we obtain that up to a $\mu$-null set \[ D_{\mu,\cT^0(X)}(x,y) = \exp(\Psi_f(x,y)).  \]  
			And thus $\mu$ is étale Gibbs for $f$.
		\end{proof}
		
		Before closing up this section, we show that the condition that $h_{\Sigma}(\Gamma \curvearrowright X)\geq 0$ cannot be removed.
		
		\begin{example}
			Let $F_2$ be the free group on two generators $a,b$ and consider the subshift of finite type $X \subset (\ZZ/4\ZZ)^{F_2}$ given by the algebraic condition \[ X = \{x \in (\ZZ/4\ZZ)^{F_2} : x(gs) = x(g)+1 \bmod{2}, \mbox{ for } s \in \{a,b\}    \}.   \]
			This action is topologically conjugate to direct product of the full $F_2$-shift on two symbols times the action on $\ZZ/2\ZZ$ on which both generators act non-trivially. It is known~\cite[Theorem 4.1]{Bowen_2020} that if $\Sigma$ is a sofic approximation sequence which is ``far from bipartite'' (for instance one obtained by choosing the homomorphisms $\sigma_i\colon F_2 \to \Sym(V_i)$ to be generated by $\sigma_i(a)$ and $\sigma_i(b)$ and choosing those uniformly at random) then $F_2 \curvearrowright \ZZ/2\ZZ$ admits no models for large enough $i$ and thus $h_{\Sigma}(\Gamma \curvearrowright X) = -\infty$. In this setting, every invariant measure is an equilibrium measure for $\Sigma$ (with respect to $f=0$) but the unique Gibbs measure of the system is the product of the uniform Bernoulli measure on the full $F_2$-shift and the uniform measure on $\ZZ/2\ZZ$.
		\end{example}
		
		\section{Lanford--Ruelle theorem beyond locally constant functions}\label{sec:convexgeneralization}
		
		In this section we extend the conclusion of \Cref{thm:LR_etale_local} beyond locally constant functions, and in particular provide a proof of~\Cref{thm:LR_norm_summable_interaction}, and more generally of~\Cref{thm:LR_etale_good}. The argument is essentially an application of abstract functional analysis of convex functions on topological vector spaces. Our basic strategy here closely follows the original approach of Lanford and Ruelle (see Appendix A of~\cite{LanfordRuelle1969}), with a slight deviation which we discuss below.
		
		By a continuous norm on a topological vector space $V$ we mean a function $\|\cdot\|\colon V \to [0,\infty)$ satisfying the axioms of a norm which is continuous with respect to the topology that makes $V$ a topological vector space. Note that the topology induced on $V$ by the norm $\|\cdot\|$ need not coincide with the given topology of $V$. A typical example of a topological vector space that admits a continuous norm is $C^{\infty}([0,1])$, where the uniform norm is one example for a continuous norm.  The following  functional-analytic result is the essence of our argument:

		\begin{proposition}\label{prop:dense_tangent}
			Let $V$ be a topological vector space which admits a continuous norm. Let $\overline{V}$ denote the completion of $V$ with respect to that norm. Let $f\colon \overline{V} \to \mathbb{R}$ be a continuous convex function, and $\psi\colon V \to \overline{V}$ be a continuous function. Suppose that there exists a dense subset $A \subset V$ such that if $\bar{v} \in A$, then there exists a tangent functional $\bar{w} \in \overline{V}^*$ of $f$ at $\bar{v}$ with norm at most $1$, such that $\bar{w}(\psi(\bar{v})) \ge 0$. Then for every $v \in V$ and every tangent functional $w \in \overline{V}^*$ of $f$ at $v$ we have $w(\psi(v)) \ge 0$.
		\end{proposition}
		\begin{proof}
			The first step in the proof is to show that the set $V_0 \subset V$ consisting of those $v \in V$ for which there exists a tangent functional $w \in \overline{V}^*$ of $f$ at $v$ with norm at most $1$ such that $w(\psi(v)) \ge 0$, is closed. Since $A \subset V_0$ and $A$ is dense in $V$, it would follow that $V_0 =V$.
			
			Fix $v \in V$. Let $(v_\alpha)_{\alpha \in \mathcal{D}}$ be a net which takes values on $A$ and converges to $v$. For each $\alpha \in \mathcal{D}$, choose a tangent functional $w_\alpha \in \overline{V}^*$ of $f$ at $v_\alpha$ with norm at most $1$, such that $w_\alpha(\psi(v_\alpha)) \ge 0$. By weak-$*$ compactness of the unit ball in $\overline{V}^*$ we can pass to a subnet and assume that $(w_\alpha)_{\alpha \in \mathcal{D}}$ converges in the weak-$*$ topology to some $w \in \overline{V}^*$. Then $w$ is a tangent functional of $f$ at $v$ whose norm is bounded by $1$.
			
			Fix $\varepsilon >0$. By continuity of $\psi$, it follows that {$(\psi(v_\alpha))_{\alpha \in \mathcal{D}}$} converges to $\psi(v)$. {Since all $w_\alpha$'s have norm at most $1$, $(w_\alpha(\psi(v_\alpha))-w_\alpha(\psi(v)) )_{\alpha \in \mathcal{D}}$} converges to $0$. Furthermore, as $w_\alpha(\psi(v_\alpha)) \geq 0$, it follows that $\liminf_{\alpha}w_\alpha(\psi(v)) \ge 0$. Finally, as $(w_\alpha)_{\alpha \in \mathcal{D}}$ converges to $w$ in the weak-$*$ topology, it follows that $(w_\alpha(\psi(v)))_{\alpha \in \mathcal{D}}$ converges to $w(\psi(v))$, so $w(\psi(v)) \ge 0$ and thus $v \in V_0$. This shows that $V_0$ is closed. 

			Now fix some arbitrary $v \in V$ and suppose $\omega \in \overline{V}^*$ is a tangent functional for $f$ at $v$. Our goal is to prove that $\omega(\psi(v)) \ge 0$. Let $\varepsilon >0$ be arbitrary. It suffices to prove that $\omega(\psi(v)) + \varepsilon > 0$. Choose $v' \in V$ such that $\|\psi(v)-v'\| < \frac{\varepsilon}{2\norm{w}^*}$. Since $|\omega(\psi(v))-\omega(v')| \le \norm{w}^*\|\psi(v)-v'\| < \frac{\varepsilon}{2}$, it suffices to show that 
			$\omega(v') +\frac{\varepsilon}{2} \geq 0$.
			
			The function $F\colon \mathbb{R} \to \mathbb{R}$ given by $F(t)=f(v+tv')$ is convex. Now let $w_t \in \overline{V}^*$ be a tangent functional at $v+tv'$, it follows that for any $t' \in \RR$ we have $F(t+t')-F(t)\geq w_t(t'v')$. In particular, if $F$ is differentiable at $t$ then for every tangent $w_t$ of $f$ at $v+tv'$ we have $w_t(v')=F'(t)$. By convexity of $F$ it follows that the left derivative of $F$ at $0$, denoted by $F'(0^-)$,  satisfies $F'(t) \le  F'(0^-)$ for every $t<0$ where $F$ is differentiable, which occurs for all but a countable set of $t$'s. 
			Since $w(v') \ge F'(0^-)$ we have in 
			particular:
			
			\[  w(v') \geq \liminf_{ t\to 0^-} w_t(v')  \]
			
			Using the result from the first part of the proof, for any $t \in \RR$ we may choose a tangent $w_t$ of $f$ at $v+ tv'$ with norm at most $1$, such that $w_t(\psi(v+tv')) \ge 0$. We have the following estimates:

			\[|w_t(\psi(v+tv')) - w_t(v')| \le \|\psi(v+tv')- v'\| \le \|\psi(v+tv')-\psi(v)\|+ \frac{\varepsilon}{2}.\]
			The  inequality on the left follows because $\norm{w_t}^*\leq 1$ for every $t \in \mathbb{R}$. As both $\psi\colon V \to \overline{V}$ and the norm $\| \cdot \|$ on $V$ are continuous, it follows that $\|\psi(v+tv')-\psi(v)\|$ converges to $0$ as $t$ converges to $0$. From here we obtain that $\liminf_{ t\to 0^-}w_t(v') +\frac{\varepsilon}{2} \geq 0$. Thus, $w(v') + \frac{\varepsilon}{2} \geq 0$.
		\end{proof}
		
		The original argument of Lanford and Ruelle from~\cite{LanfordRuelle1969} implicitly proves a rather similar statement. A fundamental difference is that in the argument involved in their article, the topological vector space $V$ has additional properties that guarantee that the convex function $f$ is differentiable at a residual set of points. Next a variant of Mazur's lemma proven in~\cite{LanfordRobinson1968} is invoked to show that every tangent functional of $f$ at any $v \in V$ is in the closed convex-hull of tangent functionals at differentiable points, from which the conclusion follows.
		In contrast, in~\Cref{prop:dense_tangent} above we do not assume much knowledge about the topological and functional-analytic properties of the topological vector space $V$ itself.
		
		In order to apply~\Cref{prop:dense_tangent} to our setting, we need the following characterization of measures which are \'etale Gibbs:
		
		\begin{lemma}\label{lem:etale_gibbs_reformulation}
			Let $f\colon X \to \mathbb{R}$ be such that the $\mathcal{T}(X)$-cocycle $\Psi_f$ is well defined and continuous.
			A probability measure $\mu$ is \'etale Gibbs for $f$ if and only if for
			any homeomorphism $\gamma\colon X \to X$ such that $(x,\gamma(x)) \in \mathcal{T}^0(X)$ for $\mu$-almost every $x \in X$, and any $h \in C(X)$  we have
			\[\int_X  h(\gamma(x)) - \exp(\Psi_f(x,\gamma(x)))h(x)\dd \mu \ge 0.\]
		\end{lemma}
		\begin{proof}
			On the one hand, if $\mu$ is \'etale Gibbs for $f$, it follows that \[ D_{\mu,\cT^0(X)}(x,y) = \exp(\Psi_f(x,y)) \mbox{ up to a $\mu$-null set}.   \]
			Therefore for any homeomorphism $\gamma$ of $X$ which preserves $\cT^0(X)$ and $h \in C(X)$ we have \[ \int_{X} h(\gamma(x)) \dd \mu  = \int_{X} h(x) \dd (\mu \circ \gamma) = \int_X \exp(\Psi_f(x,\gamma(x))) h(x) \dd \mu. \]
			From here the inequality follows (as an equality). Conversely, replacing $h$ by $-h$, for every $h \in C(X)$ and homeomorphism $\gamma$ of $X$ which preserves $\cT^0(X)$, we obtain the equality \[\int_X  h(\gamma(x)) \dd \mu = \int_X \exp(\Psi_f(x,\gamma(x)))h(x)\dd \mu.\]
			Because $\mu$ is a Borel measure on $X$, the Radon-Nikod\'ym derivative of a Borel measure $\nu \ll \mu$ is characterized (up to a $\mu$-null set) by the property that $\int h d\nu = \int h \frac{d \nu}{d \mu} d\mu$ for any  $h \in C(X)$.
			It follows that for any homeomorphism $\gamma$ of $X$ which preserves $\cT^0(X)$ we have $\frac{d \mu \circ \gamma}{d \mu}(x) = \exp(\Psi_f(x,\gamma(x)))$ for $\mu$-almost every $x \in X$. Because the equivalence relation $\cT^0(X)$ is the orbit equivalence of the  homeomorphisms $I_{p,q}$, it follows that $\exp(\Psi_f(x,y))$ is the Radon-Nikod\'ym cocycle of $\mu$ with respect to the equivalence relation $\cT^0(X)$ and thus that $\mu$ is \'etale Gibbs for $f$.
		\end{proof}
	
		Let us recall that the metric $\rho_{\mathcal{T}^0(X)}$ is defined on all $f,g \in C(X)$ such that the $\mathcal{T}^0(X)$-cocycles $\Psi_f,\Psi_g$ are well defined and continuous. It is given by 
		\[
		\rho_{\mathcal{T}^0(X)}(f,g)=\|f-g\|_\infty + d_{\mathcal{T}^0(X)}(\Psi_f,\Psi_g).\]
		Where $d_{\mathcal{T}^0(X)}$ is the metric from~\Cref{def:metric_T_cocycles}.
		
		Th definition of $\rho_{\mathcal{T}^0(X)}$ implicitly sets the stage for a space of potentials where our extension scheme applies.

		\begin{theorem}\label{thm:LR_etale_good_bis}\textbf{(\Cref{thm:LR_etale_good})}
			Let $\Sigma$ be a sofic approximation sequence for $\Gamma$, $X$ be a subshift such that $h_{\Sigma}(\Gamma \curvearrowright X)\geq 0$, and let $f \in C(X)$ be a  $\rho_{\mathcal{T}^0(X)}$-limit of locally constant functions. Then any equilibrium measure $\mu$ on $X$ for $f$ with respect to $\Sigma$ is \'etale Gibbs with respect to $f$. 
		\end{theorem}
		
		\begin{proof}
			Let $V$ denote the space of functions $f \in C(X)$ that are  $\rho_{\mathcal{T}^0(X)}$-limit of locally constant functions. Then $\rho_{\mathcal{T}^0(X)}$ induces a metric on $V$ which makes it a topological vector space. It is clear that the locally constant functions are dense in $V$ and that the uniform norm is continuous on $V$. Since the locally constant functions are also dense in $C(X)$, it follows that the metric completion $\overline{V}$ of $V$ with respect to the uniform norm is (up to a natural isomorphism) $C(X)$. Also the map $f \mapsto \Psi_f$ is well defined on $V$ and defines a continuous linear function from $V$ to the space of continuous $\mathcal{T}^0(X)$-cocycles.
			
			By~\Cref{prop:tangent_functionals_eq_measures}, the map $\Pi \colon C(X) \to \RR$ given by $\Pi(f) = P_{\Sigma}(\Gamma \curvearrowright X,f)$ is a convex continuous map and the equilibrium measures for $f \in C(X)$ are all tangent functionals of $\Pi$ at $f$. 
			Choose $f_0\in V$ and let $\gamma$ be a homeomorphism of $X$ such that $(x,\gamma(x))\in \cT^0(X)$ for every $x \in X$, and let $h \in C(X)$.
			
			Let $A$ be the subspace of $V$ given by the locally constant functions and $\psi\colon V \to C(X)$ be given by
			$\psi(f)(x) = h(\gamma(x))-\Psi_f(x,\gamma(x))h(x)$. Continuity of $\psi$ follows from continuity of the map $f \mapsto \Psi_f$ on $V$.
			Since any subshift is expansive the measure-theoretic sofic entropy is an upper semi-continuous function of the measure (as in \cite[Theorem $2.1$]{chungZhang205weakExpansiveness}). In particular for any $f\in C(X)$ there exists an equilibrium measure on $X$ (with respect to $\Sigma$). Fix $f_1\in A$, by~\Cref{thm:LR_etale_local}, it follows that any equilibrium measure $\mu_1$ at $f_1$ is \'etale Gibbs with respect to $\Sigma$. By~\Cref{prop:tangent_functionals_eq_measures} any  equilibrium measure  at $f_1$ is a tangent functional $\mu_1$ for the topological sofic pressure function $\Pi\colon C(X) \to [-\infty,+\infty]$. By~\Cref{lem:etale_gibbs_reformulation} this implies that an equilibrium measure $\mu_1$ at $f_1$ satisfies   \[ \mu_1(\psi(f_1)) = \int_X h(\gamma(x))-\Psi_{f_1}(x,\gamma(x))h(x) \dd \mu_1 \geq 0.  \]

			Applying~\Cref{prop:dense_tangent}, it follows that for every $f_0 \in V$ and every tangent functional $\mu$ for $\Pi$ at $f_0$ we have \[ \mu(\psi(f_0)) = \int_X h(\gamma(x))-\Psi_{f_0}(x,\gamma(x))h(x) \dd \mu \geq 0.  \]
			
			It follows, again by~\Cref{lem:etale_gibbs_reformulation}, that every equilibrium measure $\mu$ at $f_0 \in V$ is \'etale Gibbs for $f_0$ with respect to $\Sigma$.
		\end{proof}
		
		\subsection{Absolutely summable interactions}\label{subsec:ASinteractions}
		
		\begin{definition}
			An \define{interaction} on a space of configurations $X$ is a real valued function on the $X$-admissible patterns $P(X)$. Given an interaction $\Phi\colon P(X) \to \RR$ and $F \Subset \Gamma$ we denote by $\Phi_F$ the restriction of $\Phi$ to $P_F(X)$. 
			An interaction $\Phi$ on $X$ a shift space $X$ is \define{translation-invariant} if $\Phi(p)= \Phi(g(p))$ for every $g \in \Gamma$ and $p \in P(X)$.
			An interaction $\Phi$  is called \define{absolutely-summable} (also called \define{norm-summable}) if 
			\[ \| \Phi\| \isdef \sum_{ 1_\Gamma \in F \Subset \Gamma} \| \Phi_F\|_\infty < +\infty ,\]
			where $\| \Phi_F\|_\infty = \max_{p \in P_F(X)} |\Phi(p)|$ is the uniform norm on the (finite-dimensional) space $\mathbb{R}^{P_F(X)}$. 
		\end{definition}
		
		A translation-invariant interaction $\Phi$ is of \define{finite-range} if there exists $F_0 \Subset \Gamma$ such 
		that $\Phi_F$ is identically zero whenever $F$ is not contained in $gF_0$ for some $g \in \Gamma$. When $\Gamma$ is a finitely generated group, this is equivalent to the statement that $\Phi$ is supported on patterns whose supports have bounded diameter.
		
		The space of absolutely-summable interactions is a Banach space with respect to the above norm~\cite[Proposition $2.19$]{barbieri2021Gibbsianrep}. The subspace of finite-range interactions is dense.

		We now discuss a procedure to obtain a continuous map $f_\Phi\colon X \to \RR$ from a absolutely-summable interaction $\Phi$, following Ruelle~\cite{Rue04} (in the context of $\ZZ^d$-subshifts). 
		Let $\mathcal{P}_{c}(\Gamma)$ denote the space  of finite non-empty subsets of $\Gamma$.
		Let  $\widetilde{\mathcal{P}_c}(\Gamma)\subset \mathcal{P}_{c}(\Gamma)$ be a subset containing exactly one  representative in each $\Gamma$-orbit.
		Namely, for every $F \in \mathcal{P}_{c}(\Gamma)$ the set $\{g F: g \in \Gamma\}$ contains exactly one element of $\widetilde{\mathcal{P}_c}(\Gamma)$. Let $\operatorname{Fix}_{\Gamma}(F)$ be the set of $g \in \Gamma$ such that $gF = F$.
		Define $f_{\Phi}\colon X \to \RR$ as \[ f_{\Phi}(x) \isdef \sum_{ F \in \widetilde{\mathcal{P}_c}(\Gamma)} \frac{1}{|\operatorname{Fix}_{\Gamma}(F)|}\Phi(x|_{F}).\]
		
		Let us remark that in the case where $\Gamma$ has no torsion (as in the case of $\ZZ^d$ in the work of Ruelle), then $\operatorname{Fix}_{\Gamma}(F)$ is trivial and can be removed from the equation above.

		An alternative, and perhaps more natural, way of defining the potential is to just average over all translates of a finite set which contain the identity, namely, we can define $h_\Phi\colon X \to \mathbb{R}$ by
		\[ h_\Phi(x) \isdef \sum_{1_\Gamma \in F \Subset \Gamma} \frac{1}{|F|}\Phi(x|_F).\]
		
		Let us notice that these two definitions are in general different, and it is not even clear that they define the same class of functions. However, they do define the same $\cT(X)$-cocycle whenever it is well-defined.  Indeed,
		
		\begin{align}
			\Psi_{f_{\Phi}}(x,y) & = \sum_{g \in \Gamma}\left( f_\Phi(gy)- f_\Phi(gx)     \right) \\
			& = \sum_{g \in \Gamma} \sum_{F \in \widetilde{\mathcal{P}_c}(\Gamma)} \frac{1}{|\operatorname{Fix}_{\Gamma}(F)|}\left(\Phi(y|_{g^{-1}F})-\Phi(x|_{g^{-1}F})\right) \\
			& = \sum_{F \Subset \Gamma} \left( \Phi(y|_F)-\Phi(x|_F)\right)\\
			& = \sum_{g \in \Gamma} \sum_{g \in F \Subset \Gamma} \frac{1}{|F|}\left(\Phi(y|_{F})-\Phi(x|_{F})\right) \\
			& = \sum_{g \in \Gamma} \sum_{1_\Gamma \in F \Subset \Gamma} \frac{1}{|F|}\left(\Phi(y|_{g^{-1}F})-\Phi(x|_{g^{-1}F})\right) \\
			& = \sum_{g \in \Gamma}\left( h_\Phi(gy)- h_\Phi(gx)     \right) = \Psi_{h_{\Phi}}(x,y)
		\end{align}
		
		Consequently, any Gibbs measure for $f_{\Phi}$ is a Gibbs measure for $h_{\Phi}$ and vice versa. In what follows we will use Ruelle's definition, $f_{\Phi}$, because it has the advantage that all locally constant maps can be obtained from finite-range interactions.

		Given any interaction $\Phi$ for which $f_\Phi$ can be defined, we will say that a measure on $X$ is (\'etale) Gibbs with respect to $\Phi$ if it is (\'etale) Gibbs with respect to $f_\Phi$. This definition is consistent with the original definition of Lanford and Ruelle~\cite{LanfordRuelle1969}.
		
		\begin{remark}
			Any finite-range interaction $\Phi$ induces a locally constant function $f_{\Phi}\colon X \to \mathbb{R}$. Therefore~\Cref{thm:LR_etale_local} holds for finite-range interactions.
		\end{remark}
		
		The potential $f_\Phi$ is continuous. Furthermore, the map $\Phi \mapsto f_\Phi$ is a bounded linear map from the Banach space of absolutely-summable interactions to the space of continuous functions with the uniform norm. 
		
		To simplify the notation, let us denote the $\cT(X)$-cocycle by $\Psi_\Phi$ instead of $\Psi_{f_\Phi}$ and notice that for every $(x,y)\in \cT(X)$ we have \[  \Psi_\Phi(x,y) = \sum_{F \Subset \Gamma} \left( \Phi(y|_F)-\Phi(x|_F)\right).\]
		Whenever $\Phi$ is absolutely-summable and translation-invariant, one can show that for any $(x,y)\in \cT_F(X)$ we have that $|\Psi_{\Phi}(x,y)| \leq 2|F| \| \Phi \|$ (see Proposition 3.1 of~\cite{barbieri2021Gibbsianrep}) and therefore $\Psi_{\Phi}$ is a well-defined $\Gamma$-invariant $\cT(X)$-cocycle. A direct computation shows that $\Psi_{\Phi}$ is continuous.
		
		Let
		\[ \texttt{NS}(X) \isdef \left\{ f_\Phi : \Phi \mbox{ is an absolutely-summable, translation-invariant interaction on } X\right\}.\]
		
		Let us endow $\texttt{NS}(X)$ with the norm given by 	\[ \|f \|_{\texttt{NS}} \isdef \inf\{ \| \Phi\| :~ f= f_\Phi \}.\]
		
		It is straightforward to verify that $\norm{\cdot}_{\texttt{NS}}$ is a semi-norm. The fact that it is a norm follows from the fact that $\norm{f_{\Phi}}_{\infty} \leq \norm{f_{\Phi}}_{\texttt{NS}}$. which is direct from the definition of $f_{\Phi}$.
		
		First, let us show that $f_{\Phi} \mapsto \Psi_\Phi$ is a continuous linear map from $\texttt{NS}(X)$ to the topological vector space of continuous, translation-invariant $\mathcal{T}^0(X)$-cocycles. Let $f \in \texttt{NS}(X)$ and choose $\Phi$ such that $f=f_{\Phi}$. Let also $F\Subset \Gamma$. By Proposition 3.1 of~\cite{barbieri2021Gibbsianrep} we have that for any $(x,y)\in \cT_F(X)$ we have \[ |\Psi_{f}(x,y)| = |\Psi_{\Phi}(x,y)| \leq 2|F|\norm{\Phi}. \]
		In particular, the bound is still valid for $(x,y)\in \cT^0_F(X)$. Taking infimum over all $\Phi$ such that $f=f_{\Phi}$ we get that \[ |\Psi_{f}(x,y)| \leq 2|F|\norm{f}_{\texttt{NS}} \mbox{ for every } (x,y)\in \cT^0_F(X).  \] Therefore the map $f_{\Phi} \mapsto \Psi_\Phi$ is continuous.

		\begin{proposition}\label{prop:NS_isLRgood}
			Any $f \in \texttt{NS}(X)$ is a $\rho_{\mathcal{T}^0(X)}$-limit of locally constant functions.
		\end{proposition}
		\begin{proof}
			Let $\Phi$ be a translation-invariant,  absolutely summable interaction with $f = f_{\Phi}$.  Let $(F_n)_{n \in \NN}$ be an enumeration of all finite subsets of $\Gamma$ modulo translation. For each $n \geq 1$ consider the interaction $\Phi^{(n)}$ which coincides with $\Phi$ on all patterns whose supports is a translation of $F_k$ for $k \leq n$, and is $0$ everywhere else. It follows that $\Phi^{(n)}$ is a finite-range, translation-invariant interaction and that $f_n := f_{\Phi^{(n)}}$ is  a locally constant function. To complete the proof, it suffices to show that $\lim_{n \to \infty}\rho_{\mathcal{T}^0(X)}(f_n,f)=0$.
			Indeed, as $\Phi$ is absolutely-summable, we have that \[\lim_{n \to \infty}\norm{\Phi-\Phi^{(n)}} = \sum_{k=n+1}^\infty\norm{\Phi_{F_k}}_\infty =0.\]
			From this it follows that 
			\[d_{\mathcal{T}^0(X)}(\Psi_{f_n},\Psi_{f}) \le \sum_{j=1}^\infty \frac{1}{2^j}\min\left\{1,2|F_j| \norm{\Phi-\Phi^{(n)}}\right\} \to 0 \mbox{ as }n \to \infty.\]
			Also, $\norm{f- f_n}_\infty \le \norm{f- f_n}_{\texttt{NS}} \to 0$ as $n \to \infty$.
			This shows that  $\lim_{n \to \infty}\rho_{\mathcal{T}^0(X)}(f_n,f)=0$.\end{proof}
		
		\begin{theorem}\label{thm:LR_norm_summable_interaction_bis}\textbf{(\Cref{thm:LR_norm_summable_interaction})}
			Let $\Sigma$ be a sofic approximation sequence for $\Gamma$, $X$ be a subshift that satisfies the topological Markov property such that $h_{\Sigma}(\Gamma \curvearrowright X)\geq 0$, $\Phi$ an absolutely-summable interaction on $X$ and $\mu$ an equilibrium measure on $X$ for $\Phi$ with respect to $\Sigma$. Then $\mu$ is Gibbs with respect to $\Phi$.
		\end{theorem}
		
		\begin{proof}
			The proof follows from putting together~\Cref{thm:LR_etale_good},~\Cref{prop:NS_isLRgood} and~\Cref{prop:GibbsTMP}. Indeed, by~\Cref{prop:NS_isLRgood} we have that any map induced by a translation-invariant, absolutely-summable interaction is a $\rho_{\mathcal{T}^0(X)}$-limit of locally constant functions, and thus from~\Cref{thm:LR_etale_good} it follows that every equilibrium measure is étale Gibbs. Finally, as the space satisfies the topological Markov property, \Cref{prop:GibbsTMP} tells us that $\cT(X)=\cT^0(X)$ and thus we may replace ``étale Gibbs'' by ``Gibbs''.
		\end{proof}
		
		\subsection{Functions with $\mathbb{F}$-summable variation}\label{subsec:FSvariation}
		
		Some authors rather than using a space of interactions, directly use a space of continuous functions $f\colon X \to \RR$ to model ``the potential energy'', for instance~\cite{Meyerovitch_2013,Kellerbook}. In this setting, we introduce the notion of functions with summable variation with respect to a \define{filtration} of $\Gamma$, that is, an increasing sequence of finite subsets of $\Gamma$ which cover it. This notion of convergence generalizes the concept of $d$-summable variation of~\cite{Meyerovitch2017}, the notion of ``regular local energy functions'' of~\cite{Kellerbook} and the notion of ``shell-regular potentials'' of~\cite{BorMac_2020}. We show that every function with this property satisfies the equivalent statement of the Lanford--Ruelle theorem.

		We now define a class of potentials which generalizes the space of functions having ``$d$-summable variation'' in the case $\Gamma = \ZZ^d$ as in \cite{Meyerovitch_2013} and show that they are $\rho_{\mathcal{T}^0(X)}$-limits of locally constant functions, thus recovering \cite[Theorem $3.1$]{Meyerovitch_2013} and generalizing it the case where the acting group is an arbitrary sofic group.
		
		\begin{definition}
			We call a sequence $\mathbb{F}= (F_n)_{n=1}^\infty$ of finite subsets of $\Gamma$ a \define{filtration} of $\Gamma$ if $F_n \subset F_{n+1}$ for all $n$ and $\Gamma = \bigcup_{n=1}^\infty F_n$.
			Given $f\colon X \to \mathbb{R}$ and $S \Subset \Gamma$ denote
			\[\Var_{S}(f) \isdef \sup\left\{ |f(x)-f(y)|:~ x,y \in X \mbox{ and } x|_{S} = y|_{S} \right\},\]
			
			Given a filtration $\mathbb{F}=(F_n)_{n=1}^\infty$ of $\Gamma$ and $S \Subset \Gamma$, let 
			\[\norm{f}_{\texttt{SV}(\mathbb{F}),S} \isdef  \sum_{n=1}^\infty |F_{n+1}S\setminus F_nS| \Var_{F_{n}}(f). \]
			
			A function $f\colon X \to \mathbb{R}$ has \define{$\mathbb{F}$-summable variation} if $\norm{f}_{\texttt{SV}(\mathbb{F}),S} < \infty$ for all $S \Subset \Gamma$.
			We denote the space of functions on $X$ with $\mathbb{F}$-summable variation by $\texttt{SV}_{\mathbb{F}}(X)$.
		\end{definition} 
		
		One may verify that any function with $\mathbb{F}$-summable variation is continuous, and that any locally constant function has $\mathbb{F}$-summable variation with respect to any filtration of $\Gamma$. We endow the vector space $\mathtt{SV}_{\mathbb{F}}(X)$ with the topology induced by the countable collection of semi-norms $(\norm{\cdot}_{\texttt{SV}(\mathbb{F}),S})_{S \Subset \Gamma}$ together with the uniform norm $\|\cdot\|_\infty$. In general, $\mathtt{SV}_{\mathbb{F}}(X)$ is not a Banach space, but it may be verified that it is always a Fr\'echet space (we shall not make use of this property).
		
		Particularly, when $\Gamma= \ZZ^d$ and $\mathbb{F}= (F_n)_{n=1}^\infty$ is the filtration of $\ZZ^d$ given by $F_n= \{-n,\ldots,n\}^d$, then
		$\norm{f}_{\mathtt{SV}(\mathbb{F}),S} < \infty$ for all $S \Subset \ZZ^d$ if and only if
		$\sum_{n=1}^\infty n^{d-1} \Var_{F_n}(f) < +\infty$, so in that case $\mathtt{SV}_{\mathbb{F}}(X)$ is exactly the space of function with $d$-summable variation as defined in~\cite{Meyerovitch_2013} or the space of ``regular local energy functions'' as defined in~\cite[Section 5]{Kellerbook}. More generally, if $\Gamma$ is a finitely generated group with ``bounded sphere ratios'' and $\mathbb{F}$ is the filtration corresponding to balls with respect to some finite symmetric generating set, then $\mathtt{SV}_{\mathbb{F}}(X)$ is exactly the space of ``shell-regular potentials'' as defined in~\cite[Section 5]{BorMac_2020}.

		For the rest of this section let $\mathbb{F} = (F_n)_{n=1}^\infty$ be an arbitrary filtration on the countable group $\Gamma$. We shall prove that any  $f \in \mathtt{SV}_{\mathbb{F}}(X)$ is a $\rho_{\mathcal{T}^0(X)}$-limit of locally constant functions. It it is clear that convergence in $\mathtt{SV}_{\mathbb{F}}(X)$ implies $\norm{\cdot}_\infty$-convergence, so we only need to verify that locally constant functions are dense, and that the map $f \mapsto \Psi_f$ from $\mathtt{SV}_{\mathbb{F}}(X)$ to the space of continuous  $\cT^0(X)$-cocycles is continuous. Let us mention that these two proofs extend almost verbatim from the well known ``classical case'' $\Gamma= \ZZ^d$, and $\mathbb{F} = (\{-n,\ldots,n\})^d)_{n=1}^\infty$.
		
		\begin{lemma}\label{lem:SVF_local_dense}
			The set of locally constant functions is dense in $\mathtt{SV}_{\mathbb{F}}(X)$.
		\end{lemma}
		
		\begin{proof}
			Let $f \in \mathtt{SV}_{\mathbb{F}}(X)$. For every $n \geq 1$ consider the locally constant function
			$f_n\colon X \to \mathbb{R}$ given by
			\[f_n(x) \isdef \sup\{ f(y):~ y|_{F_n} = x|_{F_n}\}.\]
			We claim that $(f_n)_{n =1}^{\infty}$ converges to $f$ in $\mathtt{SV}_{\mathbb{F}}(X)$. Indeed, as $f$ is continuous it is clear that $\norm{f-f_n}_{\infty}$ converges to $0$. Therefore it suffices to check that $\norm{f-f_n}_{\mathtt{SV}(\mathbb{F}),S}$ converges to $0$ for every $S \Subset \Gamma$. This follows directly from the fact that $\Var_{F_m}(f-f_n)\leq 2\Var_{F_n}(f)$ for every $m \leq n$.
		\end{proof}
		
		\begin{lemma}\label{lem:SV_cocycle_cont}
			
			For every function $f\colon X \to \mathbb{R}$ with $\mathbb{F}$-summable variation, $S \Subset \Gamma$ and $(x,y) \in \cT_S(X)$, the series defining $\Psi(x,y)$
			in~\Cref{defn:Psi_f} is absolutely convergent. Hence, $\Psi_f$ is a well-defined $\Gamma$-invariant $\cT(X)$-cocycle and is furthermore continuous. Moreover, the map $f \mapsto \Psi_f$ from $\mathtt{SV}_{\mathbb{F}}(X)$ to the space of continuous  $\cT^0(X)$-cocycles is continuous. Consequently, convergence in $\mathtt{SV}_{\mathbb{F}}(X)$ implies $\rho_{\mathcal{T}^0(X)}$-convergence.
		\end{lemma}
		\begin{proof}
			Let $f\in \mathtt{SV}_{\mathbb{F}}(X)$, $S \Subset \Gamma$ and $(x,y) \in \cT_{S}(X)$. Let us write $F_0 = \varnothing$. As $\mathbb{F}$ is a filtration we may write $\Gamma$ as the disjoint union \[ \Gamma = \bigcup_{n=0}^{\infty}(F_{n+1}S^{-1} \setminus F_nS^{-1}).   \]
			Therefore we may write
			\[\sum_{ g\in \Gamma} |f(gy)-f(gx)| =\sum_{n=0}^\infty\sum_{g \in F_{n+1}S^{-1} \setminus F_nS^{-1}} |f(gy)-f(gx)|. \]
			Now for every $g \in \Gamma$, we have the trivial bound $|f(gy)-f(gx)| \leq 2\|f\|_\infty$. Moreover, as $(x,y) \in \cT_{S}(X)$ it follows that if $g \in \Gamma \setminus F_nS^{-1}$ then $(gx)|_{F_n} = (gy)|_{F_n}$ and thus $|f(gy)-f(gx)| \leq \Var_{F_n}(f)$, so
			
			\begin{align}
				\sum_{g \in \Gamma} |f(gy)-f(gx)| & \leq 2|F_1S^{-1}|\norm{f}_{\infty} + \sum_{n=1}^{\infty} |F_{n+1}S^{-1} \setminus F_nS^{-1}| \Var_{F_n}(f)\\ & = 2|F_1S^{-1}|\norm{f}_{\infty} + \norm{f}_{\mathtt{SV}(\mathbb{F}),S^{-1}}.
			\end{align}
			We conclude that the series defining $\Psi_f(x,y)$ is indeed absolutely convergent for every $S \Subset \Gamma$ and $(x,y) \in \cT_{S}(X)$.

			Let us now show that $\Psi_f\colon\mathcal{T}(X) \to \mathbb{R}$ is continuous. As explained in~\Cref{sec:etale}, it suffices to show that $\Psi_f|_{\cT_S(X)}$ is continuous for every $S\Subset \Gamma$. More explicitly, we need to show that for every $S \Subset \Gamma$ and $\varepsilon >0$ there exists $K \Subset \Gamma$ such that if $(x,y),(x',y') \in \cT_S(X)$ are such that $x|_K = x'|_K$ and $y|_K = y'|_K$ then $|\Psi_f(x,y)-\Psi_f(x',y')| \leq \varepsilon$.
			
			Fix $S\Subset \Gamma$ and $\varepsilon >0$. Let $N,M \in \NN$ and choose $K\Subset \Gamma$ such that $SF_N^{-1}F_M \Subset K$. Let $(x,y),(x',y') \in \cT_S(X)$ such that $x|_K = x'|_K$ and $y|_K = y'|_K$. Then it follows that for every $g \in F_NS^{-1}$ we have $g(x)|_{F_M} = g(x')|_{F_M}$ and $g(y)|_{F_M} = g(y')|_{F_M}$ and thus $|f(gy)-f(gy')| \leq \Var_{F_M}(f)$ and $|f(x)-f(x')| \leq \Var_{F_M}(f)$. Since $(x,x') \in \mathcal{T}_S(X)$ and $(y,y') \in \mathcal{T}_S(X)$, for every $n \in \NN$ and $g \in \Gamma \setminus F_{n}S^{-1}$ we have $|f(gy)-f(gx)| \leq \Var_{F_n}(f)$ and $|f(gy')-f(gx')| \leq \Var_{F_N}(f)$.
			
			Putting all the above bounds together yields:
			
			\begin{align}
				|\Psi_f(x,y)-\Psi_f(x',y')| & \leq \sum_{g \in F_NS^{-1}} \left(|f(gy)-f(gy')|+|f(gx)-f(gx')|\right) \\ & \quad\quad + \sum_{n=N}^\infty \sum_{g \in F_{n+1}S^{-1}\setminus F_{n}S^{-1}}\left(|f(gy)-f(gx)|+|f(gy')-f(gx')|\right)\\
				& \leq 2|F_NS^{-1}|\Var_{F_M}(f) + 2\sum_{n=N}^\infty |F_{n+1}S^{-1}\setminus F_{n}S^{-1}| \Var_{F_n}(f).
			\end{align}
			
			Since $f \in \mathtt{SV}_{\mathbb{F}}(X)$, there exists $N \in \NN$ such that 
			$2\sum_{n=N}^\infty |F_{n+1}S^{-1}\setminus F_{n}S^{-1}| \Var_{F_n}(f) \leq \frac{\varepsilon}{2}$, and there exists $M \in \NN$ such that $2|F_NS^{-1}|\Var_{F_M}(f) \leq \frac{\varepsilon}{2}$. Now for any $K\Subset \Gamma$ satisfying $SF_N^{-1}F_M \subset K$, it follows that if $(x,y),(x',y') \in \cT_S(X)$ are such that $x|_K = x'|_K$ and $y|_K = y'|_K$ then $|\Psi_f(x,y)-\Psi_f(x',y')| \leq \varepsilon$. This proves that $\Psi_f$ is continuous.

			Finally, we show that the map $f \mapsto \Psi_f$ from $\mathtt{SV}_{\mathbb{F}}(X)$ to the space of continuous  $\cT^0(X)$-cocycles is continuous. Since $f \mapsto \Psi_f$ is linear, it suffices to prove continuity at $0$ for every restriction of the domain of the image to $\cT^0_S(X)$. This follows directly from the estimate 
			\[ |\Psi_f(x,y) | \leq 2|F_1S^{-1}| \norm{f}_\infty + \norm{f}_{\mathtt{SV}(\mathbb{F}),S^{-1}} \mbox{ for every } (x,y)\in \cT^0_S(X).\]
		\end{proof}
		
		\begin{theorem}\label{thm:LR_SV}
			Let $\Sigma$ be a sofic approximation sequence for $\Gamma$, $X$ be a subshift with the topological Markov property such that $h_{\Sigma}(\Gamma \curvearrowright X)\geq 0$, $f\colon X \to \mathbb{R}$ a function with $\mathbb{F}$-summable variation with respect to a filtration $\mathbb{F}$ of $\Gamma$ and $\mu$ an equilibrium measure on $X$ for $f$ with respect to $\Sigma$. Then $\mu$ is Gibbs with respect to $f$.
		\end{theorem}
		
		\begin{proof}
			The combination of~\Cref{lem:SV_cocycle_cont,lem:SVF_local_dense} along with continuity of the uniform norm shows %
			any function with $\mathbb{F}$-summable variation is a $\rho_{\mathcal{T}^0(X)}$-limit of locally constant functions,
			so the result follows from~\Cref{thm:LR_etale_good} and~\Cref{prop:GibbsTMP}.
		\end{proof}
		
		\section{Applications}\label{sec:applications}

		In this brief section we present a few applications of our main theorem. Our first result concerns dynamical systems which admit a unique Gibbs measure with respect to some $f \in C(X)$. There are several results in the literature ensuring uniqueness of Gibbs measures, see for instance~\cite{Dobrushin1968,VanDenBerg1993}.
		
		\begin{theorem}\label{thm:uniquegibbsmeasure}
			Let $\Gamma$ be a sofic group and $X\subset \ag^{\Gamma}$ be a subshift with the topological Markov property which admits a unique Gibbs measure $\mu$ with respect to $f \in C(X)$ which is a $\rho_{\mathcal{T}^0(X)}$-limit of locally constant functions. For any sofic approximation sequence $\Sigma$ of $\Gamma$ such that $h_{\Sigma}(\Gamma \curvearrowright X) \geq 0$, we have that $\mu$ is translation-invariant and furthermore, it is the unique equilibrium measure on $X$ for $f$ with respect to $\Sigma$.
		\end{theorem}
		
		\begin{proof}
			Let us fix $\Sigma$ such that $h_{\Sigma}(\Gamma \curvearrowright X) \geq 0$. By~\Cref{prop:exp_upper_semi_cont_enropy} it follows that the entropy function is upper semi-continuous and thus there exists a (translation-invariant) equilibrium measure on $X$ for $f$. \Cref{thm:LR_etale_good} implies every equilibrium measure on $X$ for $f$ is Gibbs with respect to $f$ and thus we conclude that any such measure must coincide with the unique Gibbs measure $\mu$ with respect to $f$.
		\end{proof}
		
		A consequence of~\Cref{thm:uniquegibbsmeasure}, which has already been mentioned in the introduction, is uniqueness of the equilibrium measure, and thus independence of the equilibrium measure on the sofic approximation sequence, for single-site potentials over a full shift: suppose $f\colon \ag^{\Gamma} \to \RR$ is a single-site potential on $\ag^{\Gamma}$, that is, $f(x)=f(y)$ if $x(1_{\Gamma})=y(1_{\Gamma})$ (and thus we may identify $f$ with a function $f\colon \ag \to \RR$). Then there is a unique Gibbs measure $\mu_f$ with respect to $f$ which is Bernoulli and given by \[ \mu_f(\{x \in \ag^{\Gamma} : x(g) = a \}) = \frac{\exp(f(a) )}{\sum_{b \in \ag }\exp(f(b))} \mbox{ for every } g \in \Gamma.   \]
		
		\Cref{thm:uniquegibbsmeasure} yields that $\mu_f$ is indeed the unique equilibrium measure on $\ag^{\Gamma}$ for the single-site potential $f$. As mentioned in the introduction, the question of uniqueness in this setting was asked in~\cite[Question 5.4]{chung_2013} and answered via a direct argument in~\cite[Example 7]{Bowen_2020}. It also received a positive answer via the theory of Rokhlin entropy~\cite[Corollary  3.6]{Seward2019b}.
		
		\Cref{thm:uniquegibbsmeasure} can be combined with more sophisticated criteria for uniqueness of Gibbs measures to obtain uniqueness of equilibrium measures on more complicated systems. For instance, van der Berg's percolation criterion in~\cite[Corollary $1$]{VanDenBerg1993} implies that for any finitely generated sofic group, the Ising model has a unique Gibbs measure at sufficiently high temperature.~\Cref{thm:uniquegibbsmeasure} then implies uniqueness of the equilibrium measure.
		
		A second application is the existence of Gibbs measures for any subshift $X\subset \ag^{\Gamma}$ with the topological Markov property for which there is some sofic approximation sequence with nonnegative sofic topological entropy. This generalizes a result of Alpeev~\cite{Alpeev2016}.
		
		\begin{theorem}\label{thm:existenceGibbsmeasure}
			Let $\Gamma$ be a sofic group and $X \subset \ag^{\Gamma}$ be a subshift with the topological Markov property for which there exists a sofic approximation sequence $\Sigma$ of $\Gamma$ with $h_{\Sigma}(\Gamma\curvearrowright X) \geq 0$. For any $f \in C(X)$ which is a $\rho_{\mathcal{T}^0(X)}$-limit of locally constant functions there exists a Gibbs measure on $X$ with respect to $f$.
		\end{theorem}
		
		\begin{proof}
			Let us fix $\Sigma$ such that $h_{\Sigma}(\Gamma \curvearrowright X) \geq 0$.  By~\Cref{prop:exp_upper_semi_cont_enropy} the entropy function is upper semi-continuous and thus there exists an equilibrium measure $\mu$ on $X$ for $f$ with respect to $\Sigma$. By~\Cref{thm:LR_etale_good} it follows that $\mu$ must be Gibbs with respect to $f$.
		\end{proof}
		
		In what follows we consider an algebraic variant of shift spaces. Let $H$ be a finite group and consider a subshift $X\subset H^{\Gamma}$. We say that $X$ is a \define{group shift} if $X$ forms a group under the operation induced by $H$ operating pointwise on every $g \in \Gamma$. In~\cite[Proposition 5.1]{BGMT_2020} it was shown that every group shift satisfies the topological Markov property (in fact, every algebraic action satisfies the non-symbolic variant of the topological Markov property, see~\cite[Proposition 4.1]{Barbieri_Ramos_Li_2022}). Also, it is clear that for any sofic approximation sequence $\Sigma$ we have $h_{\Sigma}(\Gamma \curvearrowright X)\geq 0$ as $X$ contains a constant configuration, namely its identity configuration $e$ which satisfies $e(g)=1_{H}$ for every $g \in \Gamma$.
		
		Given a group shift, the \define{homoclinic group} $\Delta(X)$ of $X$ is the set of all $y \in X$ such that $(y,e)\in \cT(X)$. Is is not hard to check (see~\cite[Proposition 5.2]{BGMT_2020}) that translation-invariant Gibbs measures for $f \equiv 0$ in $X$ are precisely the translation-invariant Borel probability measures which are invariant under multiplication by any element of $\Delta(X)$. From the arguments above and~\Cref{thm:LR_etale_local} we obtain that equilibrium measures on $X$ are invariant under the action of $\Delta(X)$. 
		Using~\Cref{thm:uniquegibbsmeasure}, we have the following extension of~\cite[Corollary 5.4]{BGMT_2020} to groups shifts over sofic groups:

		\begin{theorem}\label{thm:groupshifts}
			Let $\Gamma$ be a sofic group and let $X\subset H^{\Gamma}$ be a group shift. Suppose the homoclinic group $\Delta(X)$ is dense in $X$, then the Haar measure on $X$ is the unique measure of maximal entropy with respect to every sofic approximation sequence $\Sigma$.
		\end{theorem}
		
		\begin{proof}
			Let $\nu$ be a Gibbs measure on $X$. By~\cite[Proposition 5.2]{BGMT_2020}, $\nu$ is invariant under multiplication by any element of $\Delta(X)$ which is dense, and thus is invariant under multiplication by any element of $X$. It follows that the unique Gibbs measure is the Haar measure. By~\Cref{thm:uniquegibbsmeasure} it follows that there is a unique measure of maximal entropy with respect to every $\Sigma$ and it coincides with the Haar measure.
		\end{proof}

		\section{Beyond sofic groups}\label{sec:beyond_sofic}
		To the knowledge of the authors, the existence of a non-sofic group is still open. Despite this fact, our main result suggests that it is reasonable to ask whether there is a meaningful version of the Lanford--Ruelle theorem that holds beyond sofic groups, perhaps for actions of all countable groups. 
		In this short final section we point out that a hypothetical consequence of such a generalization would be a positive answer to Gottschalk's conjecture stating that every injective, continuous and $\Gamma$-equivariant map on $A^{\Gamma}$ is surjective. The core of the argument is rather old, and comes from the notion of \define{intrinsic ergodicity}.
		
		To formulate a Lanford--Ruelle theorem for group actions, one needs a reasonable notion of equilibrium measure, which implicitly depends on a notion of topological pressure which satisfies a variational principle. For this theoretical application we need only use the potential $0$, and thus we shall rather speak of topological and measure-theoretic entropy.
		
		Suppose $\Gamma$ is a topological group and $\Gamma \curvearrowright (X,\mu)$ is a probability measure preserving action. A \define{measure-theoretic entropy theory} for $\Gamma$ is an assignment  \[\Gamma \curvearrowright (X,\mu) \mapsto h(\Gamma \curvearrowright X,\mu) \in [-\infty,+\infty].\] 
		which is invariant under measure-theoretic isomorphism. A \define{topological entropy theory} for $\Gamma$ is an assignment
		\[\Gamma \curvearrowright X \mapsto h(\Gamma \curvearrowright X) \in [-\infty,+\infty].\] 
		which is invariant under topological isomorphism.
		Given a measure-theoretic entropy theory, one can obtain a topological entropy theory by imposing a variational principle as follows:
		\[h(\Gamma \curvearrowright X) = \sup_{\mu \in \Prob_\Gamma(X)} h(\Gamma \curvearrowright X, \mu) .\]
		
		A topological entropy theory which satisfies a variational principle with respect to a measure-theoretic entropy theory is always monotone with respect to subsystems in the sense that for any closed (in fact, measurable) $\Gamma$-invariant set $X_0 \subset X$ we have $h(\Gamma \curvearrowright X_0) \leq h(\Gamma \curvearrowright X)$.

		A Lanford--Ruelle theorem for an entropy theory would state that for every $\Gamma$-SFT $X \subset A^\Gamma$ (hopefully, any subshift with the topological Markov property), any $\mu \in \Prob_{\Gamma}(X)$ such that $h(\Gamma \curvearrowright X) = h(\Gamma \curvearrowright X,\mu)$ would necessarily be a  Gibbs state for $f =0$. Then uniqueness of the Gibbs measure on $A^{\Gamma}$ for $f=0$ would imply that the uniform Bernoulli measure is the unique measure of maximal entropy for $A^{\Gamma}$. As this measure is not supported in any strict subset of $A^{\Gamma}$, it follows that any strict subsystem $X_0$ of $X$ has strictly lower topological entropy. This precludes the possibility of equivariantly and continuously embedding  $A^{\Gamma}$ into a proper subsystem, thus settling Gottschalk's conjecture for the group $\Gamma$.
		
		One available measure-theoretic entropy which works in any group is \define{Rokhlin entropy}~\cite{Seward2019a}. The Rokhlin entropy of a free, ergodic, probability measure preserving action $\Gamma \curvearrowright (X,\mu)$ is given by 
		\[
		h_R(\Gamma \curvearrowright X, \mu) := \inf_{\mathcal{P}}H_\mu(\mathcal{P}),
		\]
		where $\mathcal{P}$ runs over all countable generating partitions of $\Gamma \curvearrowright X$, and $H_\mu(\mathcal{P})$ is the Shannon entropy of the partition.

		It is obvious from the definitions that $h_R(\Gamma \curvearrowright X, \mu)$ is invariant under measure-theoretic isomorphism. As a consequence, for  any closed (in fact, measurable) $\Gamma$-invariant set $X_0 \subset X$ and any $\mu \in \Prob_\Gamma(X)$ satisfying $\mu(X_0)=1$ we have {$h_R(\Gamma \curvearrowright X_0, \mu|_{X_0})= h_R(\Gamma \curvearrowright X, \mu)$, since $(\Gamma \curvearrowright X_0, \mu|_{X_0})$ is measure-theoretically isomorphic to $(\Gamma \curvearrowright X, \mu)$}. 
		
		In the case $\Gamma$ is a countable amenable group,  Rokhlin entropy coincides with the usual notion for free ergodic actions. For free actions of sofic groups, Rokhlin entropy dominates the sofic entropy with respect to any sofic approximation sequence of $\Gamma$. 
		\begin{question}
			Does Rokhlin entropy satisfy a Lanford--Ruelle theorem for every group $\Gamma$?
		\end{question}
		
		It seems the answer to the question above might not be easy because any group which satisfies a Lanford--Ruelle theorem for Rokhlin entropy must also admit an essentially free ergodic action with finite nonzero Rokhlin entropy, which is an open question. For a group $\Gamma$, one can define $h_R^{\tiny{sup}}(\Gamma)$ as the supremum over all values of $h_R$ on essentially free ergodic actions with finite Rokhlin entropy. It can be proven (see~\cite{Seward2019b}) that if there is a group $\Gamma$ for which $h_R^{\tiny{sup}}(\Gamma)<+\infty$ then for every group $\Gamma'$ which is the direct product of $\Gamma$ with an infinite and locally finite group we have $h_R^{\tiny{sup}}(\Gamma')=0$. This means that for $\Gamma'$ every essentially free ergodic action has Rokhlin entropy zero, thus no Lanford--Ruelle theorem for Rokhlin entropy can hold on $\Gamma'$. On the other hand, having $h_R^{\tiny{sup}}(\Gamma)=+\infty$ for every group $\Gamma$ already is enough to settle Gottschalk's conjecture~\cite{Seward2019b} (see also~\cite{Bowen_2020}).

		\Addresses
		
		\bibliographystyle{abbrv}
		\bibliography{ref}

	\end{document}